\documentclass[journal]{IEEEtran}

\usepackage{amsmath}
\usepackage{hyperref}
\usepackage{lipsum}
\usepackage{amsfonts}
\usepackage{amssymb}
\usepackage{amsthm}
\usepackage{graphicx}
\graphicspath{{images/}}
\usepackage{epstopdf}
\usepackage[normalem]{ulem} 
\usepackage{algorithmic}
\usepackage{comment}
\usepackage[]{algorithm}
\usepackage[noabbrev]{cleveref}

\newtheorem{theorem}{Theorem}[section]
\newtheorem{proposition}[theorem]{Proposition}%
\newtheorem{lemma}[theorem]{Lemma}

\usepackage{tikz}
\usepackage{tikz-3dplot}
    \usetikzlibrary{shapes.geometric, arrows, positioning, calc, fadings, 3d}
    \usetikzlibrary{arrows, calc, decorations.pathreplacing, decorations.markings}
\usepackage{keyval}
\usepackage{bm}
\usepackage{scalefnt}
\usepackage{verbatim}

\usepackage{cite}

\tikzstyle{mynode} = [text centered, rounded corners=3mm, draw=#1, thick, fill=#1!50, minimum size=6mm, outer sep=1mm]
\tikzstyle{myrec} = [text centered, rounded corners=1mm, draw=#1, thick, fill=#1!50, minimum size=7mm, outer sep=1mm]
\colorlet{v-harmaa}{gray!50!white}
\colorlet{t-harmaa}{gray!70!black}
\definecolor{lightBlue}{RGB}{179, 217, 255}
\definecolor{mediumBlue}{RGB}{102, 153, 255}
\definecolor{darkBlue}{RGB}{0, 57, 230}
\definecolor{whiteBlue}{RGB}{230, 238, 255}
\definecolor{Spink}{RGB}{230, 255, 230}
\definecolor{Spurple}{RGB}{0, 128, 43}
\definecolor{Sgreen}{RGB}{118, 222, 119}
\definecolor{Spink2}{RGB}{118, 222, 119}
\definecolor{Sblue}{RGB}{0, 128, 43}
\tikzset{
    every node/.style={font=\small},
    block/.style = {rectangle, draw, fill=Spink, text width=1.5em, text centered, rounded corners, minimum height=2em},
    arrow/.style={line width=0.8mm,->,>=stealth, color=Spurple},
    basic_arrow/.style={line width=0.6mm,->,>=stealth},
    relu_arrow/.style={line width=0.8mm, densely dotted,->,>=stealth, color=Sgreen},
    scalar_arrow/.style={line width=0.8mm,->,>=stealth, color=Spink2},
    skip/.style={line width=0.8mm, loosely dashed,->,>=stealth, color=Sblue},
    sum/.style = {draw, circle, minimum size=5mm, node distance=1cm}
}

\newcommand{\argmin}{\mathop{\mathrm{arg\,min}}}

\DeclareMathOperator{\LMA}{\texttt{LMA}}
\DeclareMathOperator{\prox}{prox}
\DeclareMathOperator{\diag}{diag}

\newcommand{\R}{\mathbb{R}} 
\newcommand{\C}{\mathbb{C}} 

\newcommand{\mcF}{\mathcal{F}} 
\newcommand{\iF}{\mcF^{-1}} 
\newcommand{\W}{\mathcal{W}} 
\newcommand{\iW}{\W^{-1}} 

\newcommand{\mcL}{\mathcal{L}}

\newcommand{\mcC}{\mathcal{C}}

\newcommand{\dir}{\boldsymbol{\theta}} 
\newcommand{\per}{\dir^\perp} 

\newcommand{\F}{F} 
\newcommand{\Ft}{\tilde{\F}} 
\newcommand{\G}{G} 

\newcommand{\Npx}{{N_{px}}}
\newcommand{\Nsmp}{{N_\text{samples}}}

\newlength{\imSz}
\newlength{\imSkip}

\title{Robust Model-Based Iteration for Passive Gamma Emission Tomography} %

\author{Tommi Heikkil\"a, Sara Heikkinen, Riina Rimppi, Tapio Helin
\thanks{T. Heikkil\"a, S. Heikkinen and T. Helin are with the Department of Computational engineering, LUT University, Lappeenranta, Finland (emails: \textit{tommi.heikkila@lut.fi, sara.heikkinen@lut.fi, tapio.helin@lut.fi}).}
\thanks{R. Rimppi is with the Radiation and Nuclear Safety Authority (STUK), Vantaa, Finland (email: \textit{riina.rimppi@stuk.fi}).}
}

\begin{document}
\maketitle

\begin{abstract}
    Passive Gamma Emission Tomography (PGET) is an International Atomic Energy Agency -approved technique for verifying spent nuclear fuel assemblies prior to geological disposal. Reconstructing the emission and attenuation maps from PGET measurements is a nonlinear ill-posed inverse problem, currently solved with a Levenberg-Marquardt (LM) scheme that requires 10-20 iterations to achieve sufficient accuracy. We propose an accelerated iterative solver that combines the LM algorithm with a Deep Gauss-Newton step, in which a learned operator refines the update proposed by the deterministic algorithm at each iteration. A safeguard condition based on the trust-region model ensures that the accelerated iterates perform no worse than LM and retain convergence to a critical point of the regularized objective. Within this framework we compare three architectures for the learned component: an encoder-decoder-style Convolutional Neural Network, Fourier Neural Operators, and Wavelet Neural Operators. Each is trained on a small set of coarsely simulated $9\times 9$ assemblies. Experiments on simulated and real measurements from Finnish nuclear power plants show that the proposed scheme reaches LM-quality reconstructions in roughly one third of the iterations, while revealing architecture-dependent trade-offs in robustness against out-of-distribution inputs.
\end{abstract}

\begin{IEEEkeywords}
    Passive gamma emission tomography, inverse problems, model-based learning, iterative reconstruction, deep learning.
\end{IEEEkeywords}

\section{Introduction} \label{sec:intro}

\IEEEPARstart{I}{n} 2026, Finland is expected to begin the long-term disposal of spent nuclear fuel in the world’s first deep geological repository, Onkalo, located at Olkiluoto in Eurajoki \cite{posiva2024yjh}. As a key nuclear safeguards  verification measure, each nuclear fuel assembly will be examined with respect to its emission characteristics. Since late 2017, Passive Gamma Emission Tomography (PGET) has been approved by the International Atomic Energy Agency (IAEA) for the verification of spent nuclear fuel to make sure that all nuclear material remains in peaceful use \cite{white2019verification}.

Mathematically, PGET is a nonlinear inverse problem in which the spatial distributions of gamma-ray emission and attenuation within a spent nuclear fuel assembly are reconstructed from external radiation measurements. Collimated detectors record projection data that depend on the unknown emission and attenuation values of the target. The process can be formulated into a ray-tracing forward model which also includes the measurement geometry and detector response. Recovering these quantities from noisy and incomplete data is an ill-posed problem, necessitating sophisticated reconstruction algorithms and appropriate regularization to obtain stable and physically meaningful solutions.

The first PGET reconstructions of real nuclear fuel assemblies were published in 2018 \cite{belanger2018effect, white2018application}. The reconstruction method currently used by the Finnish Radiation and Nuclear Safety Authority (STUK) is based on an iterative imaging algorithm developed in a series of subsequent publications \cite{backholm2020simultaneous, virta2020fuel, virta2022improved, virta2024gamma}. This algorithm forms the basis of the method proposed here and is described in more detail in \cref{sec:PGET}. Although robust and accurate, the algorithm is computationally demanding and usually requires at least 10-20 iterations to achieve sufficiently accurate reconstructions.

The aim of this work is to accelerate the reconstruction process, thereby facilitating the large-scale measurement and verification cycle required at the disposal facility. Also, the resulting computational efficiency makes the approach well suited for integration into future statistical inference and uncertainty quantification frameworks. In recent years, data and learning-based approaches have shown excellent qualitative results in almost every field of science, including solving difficult linear and nonlinear inverse problems \cite{li2020fourier,takamoto2022pdebench}. An inherent downside of many neural networks and other machine learning architectures is their limited explainability, i.e., the method typically provides limited rigorous guarantees for the proposed solution or the output behavior. In nuclear safeguards applications, such as the verification of spent nuclear fuel assemblies prior to geological disposal, this is a notable downside.

The performance of data-driven methods is also often strongly dependent on the quality and quantity of the training data, especially when it comes to robustness against out-of-distribution inputs. Varied and high quality training data for PGET is limited, since real data mostly come from operating nuclear power plants (NPPs), where a variety of different fuel assembly geometries is used, but rarely contain outliers such as missing or replaced rods. On the other hand, high fidelity simulations which account for scattering, spectrum of energies and varying detector sensitivity are often based on Monte Carlo simulations which makes them computationally very demanding \cite{miller2018assessing, cavallini2023vanquishing}. However, the rigid and simple geometry of the nuclear fuel assemblies and physical uniformity of individual fuel rods also makes PGET a tempting testbed for low-fidelity synthetic data experiments. Compared to, for example positron or single photon emission tomography used in medical diagnostics, it is much easier to approximate the potential solution space of missing or replaced fuel rods, than the interior structure of human body.

In this paper, we propose an iterative solver for PGET, which combines the recently proposed Deep Gauss-Newton method \cite{herzberg2021graph, mozumder2021model, hauptmann2018model} to improve and accelerate the iterates of the traditional PGET solver introduced in \cite{backholm2020simultaneous}. 
Given the traditional iterates come from a Levenberg-Marquardt (LM) trust region method, we propose a safeguard condition to the iterates to guarantee convergence and to further improve the robustness, even when the training data consist of a relatively small batch of synthetic data of just single fuel assembly geometry. In our implementation, we compare recently proposed neural operator architectures in the nonlinear imaging task, namely the Fourier Neural Operators (FNO) \cite{li2020fourier} and Wavelet Neural Operators (WNO) \cite{tripura2023wavelet}. As a baseline we have a more traditional but still effective encoder-decoder style Convolutional Neural Network (CNN), inspired by the implementation in \cite{hauptmann2018model}.

The paper is organized as follows: in \cref{sec:PGET} we give a more detailed explanation of the PGET imaging, in particular the mathematical model for the measurement process, which explains design choices and results obtained. In addition, a brief description of the traditional LM minimization algorithm is also given in \cref{ssec:LMA} as it is still needed for the model-based approaches and the safeguard condition.

In \cref{sec:Deep Gauss Newton} we discuss the model-based acceleration and introduce the safeguard procedure for the iterative scheme. Our training setup is specified in \cref{ssec:training data} and the neural network and neural operator architectures are introduced in \cref{ssec:CNN,ssec:FNO,ssec:WNO}. Finally, in \cref{sec:results} the results from both simulated and real data from operating nuclear power plants are illustrated and analyzed. Conclusions are given in \cref{sec:conclusions}.

\section{Passive Gamma Emission Tomography} \label{sec:PGET}

The PGET device is a donut-shaped measurement instrument which houses two deeply collimated CZT gamma ray detector banks (each with 91  3.5mm $\times$ 3.5mm $\times$ 1.75mm detector elements) in a water-tight enclosure (see \cref{fig:PGET}). Due to a small offset between the opposing detectors, measurements from a full rotation can be interleaved to achieve an effective detector spacing of 2 mm. Along the measurement plane the collimators are narrow and rectangular to limit the incoming photon angles, whereas vertically the collimators are trapezoidal for taller field-of-view and better photon statistics from farther away where the incoming signal is weaker.

\begin{figure}
    \centering
    \includegraphics[width=0.8\linewidth]{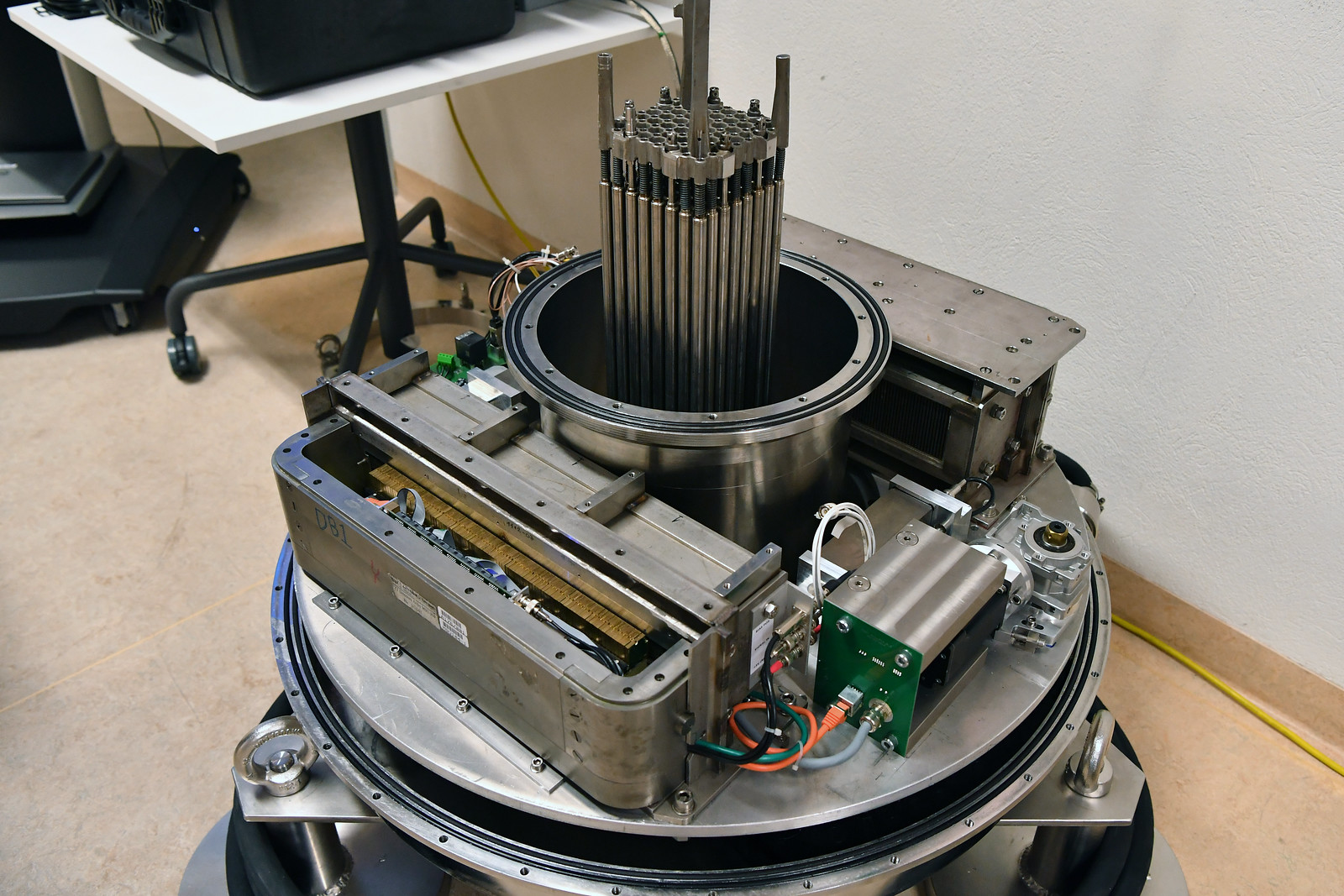}
    \caption{Dummy fuel assembly and the PGET measurement device without the cover. Photo credit: Dean Calma / IAEA.}
    \label{fig:PGET}
\end{figure}

Each measurement records the detector counts above four energy thresholds over a set time, both pre-defined by the user. The energy thresholds can be used to limit the signal contribution to specific gamma emitting isotopes such as $^{137}$Cs and $^{154}$Eu, and to rule out some of the scattered gamma rays, for example. More details on the PGET device and its development can be found in \cite{belanger2018effect,honkamaa2014prototype,white2018application}.

\subsection{Forward model} \label{ssec:forward_model}

Given the set of noisy measurements $y^\delta$, the mathematical problem is to determine both the emission values $\lambda: \Omega \mapsto \R_+$ and attenuation values $\mu: \Omega \mapsto \R_+$ within a bounded region $\Omega \subset \R^2$ containing the given fuel assembly. The assembly can contain dozens or even hundreds of highly attenuating fuel rods (spent fuel is highly emitting, whereas fresh fuel is not), possible absorber rods meant to slow down the fissile reaction at initial stages of the reactor cycle and water which is used both as coolant and neutron moderator. In the vertical direction, the materials are assumed to be uniform in emission and attenuation, but as the measurement device has a conical field of view in that direction, possible partial length rods will show as lower attenuation and emission values in the images. Therefore, we must be able to reconstruct a continuous range of values.

As in \cite{backholm2020simultaneous,virta2020fuel,virta2022improved}, our forward operator $\F$ is the attenuated Radon transform \cite{natterer2001mathematical} which models ray-tracing from every emitting location $x \in \Omega \subset \R^2$ with $\lambda(x)>0$ along a straight path to the detector at location $(s, \theta) \sim s\per + r_\Omega\dir$, where $\dir$ is the unit vector corresponding to angle $\theta$ and $r_\Omega$ is a radius large enough such that $\Omega \subset B(0,r_\Omega)$. The intensity of the emission decays exponentially according to the Beer-Lambert law, based on the linear attenuation coefficients $\mu(z), z \in \Omega$ along the ray:
\begin{align}
    &\F: L^2(\Omega) \times L^2(\Omega) \to L^2(\R \times S^1), \quad \text{where} \nonumber \\
    &\F(\lambda, \mu)(s,\theta) = y(s, \theta) \nonumber \\
    &= \int_{\R} \lambda(s\per + t\dir) e^{-B\mu(s\per + t\dir, \theta)} dt, \label{eq:fwd_model}
\end{align}
and $B\mu(x,\theta) = \int_0^\infty \mu(x + \tau \dir) d\tau$ denotes the \emph{beam transform}. 
The operator $\F$ is linear in emission $\lambda$ since the detector measures the total emission along a ray but nonlinear in attenuation $\mu$ due to the exponential decay. Despite the nonlinearity, it is a relatively well-behaving and is known to be Fr\'{e}chet differentiable. The corresponding Fr\'{e}chet derivatives are given by (cf. \cite{dicken1999new, stefanov2014identification})
\begin{align}
    D_\lambda \F(\lambda, \mu) [\delta \lambda] &= \F(\delta \lambda, \mu) \nonumber \\
    D_\mu \F(\lambda, \mu) [\delta \mu] &= - \F \big( \lambda \, B(\delta \mu), \mu \big), \label{eq:frechet_partial} 
\end{align}
where $\delta \lambda \in L^2(\Omega)$ and $\delta \mu \in L^2(\Omega)$ denote small perturbations to the emission and attenuation maps, respectively. The total derivative is then their sum
\begin{align} \label{eq:frechet}
    D \F(\lambda, \mu) [\delta \lambda, \delta \mu] &= \F(\delta \lambda, \mu) - \F\big( \lambda \, B(\delta \mu), \mu \big).
\end{align}
This shows that the derivatives are also attenuated Radon transforms with different inputs.

In particular \cite{dicken1999new} provides a detailed analysis of the smoothness properties of the attenuated Radon transform. For convenience, we show a much simplified result:
\begin{lemma} \label{lemma:F'isLipschitz}
Assume all emission and attenuation maps $\lambda_i, \mu_i \in L^2(\R^2)$ are non-negative, bounded and supported in $\Omega \subset B(0, r_\Omega)$ for some $r_\Omega > 0$, then the Fréchet derivative of $F$ is Lipschitz continuous with 
\begin{align*}
    &\| D \F(\lambda_1, \mu_1) - D \F(\lambda_2,\mu_2) \|_{\text{op}} \\
     \leq & 2\sqrt{\pi r_\Omega } (1 + C_\lambda) \left( \| \lambda_1 - \lambda_2 \|_{L^2} + \| \mu_1 - \mu_2 \|_{L^2}  \right),
\end{align*}
where $C_\lambda$ is the upper bound on $\sqrt{2 r_\Omega} \lambda(x)$.
\end{lemma}
\begin{proof}
We can write
\begin{align*}
    &\| D \F(\lambda_1, \mu_1)[\delta \lambda, \delta \mu] - D\F (\lambda_2, \mu_2)[\delta \lambda, \delta \mu] \|_2 \\[0.5em]
     & \leq \ \underbrace{ \| \F (\delta \lambda, \mu_1) - \F (\delta \lambda, \mu_2) \|_2 }_{I_1} \\
    & + \ \underbrace{ \| \F (\lambda_2 B(\delta \mu), \mu_2) - \F (\lambda_1 B(\delta \mu), \mu_1) \|_2 }_{I_2}.
\end{align*}
First we need an estimate for the exponential. For any $t \in \R, e^t \geq 1 + t$. Without loss of generality assume $u > v \geq 0$, then
\begin{align} \label{eq:expDiff}
    |e^{-u} - e^{-v}| &= e^{-v} - e^{-u} = e^{-v}(1 - e^{v - u}) \nonumber \\
    &\leq 1 - e^{v-u} \leq u-v.
\end{align}
We also need an estimate for the beam transform which is based on the bounded support of the functions. For any $\theta \in [0, 2\pi)$ and $s$, denote $\gamma(t) = \gamma_{s,\theta}(t) = s\per + t\dir$. Then
\begin{align} \label{eq:beamBound}
    \int_{-r_\Omega}^{r_\Omega} |B \mu (\gamma(t), \theta) |^2 dt &\leq \int_{-r_\Omega}^{r_\Omega}\int_0^{2r_\Omega} | \mu (\gamma(t) + \tau \theta) |^2 d\tau dt \nonumber \\
    &\leq 2r_\Omega \int_{-r_\Omega}^{r_\Omega} |\mu (\gamma(t))|^2 dt,
\end{align}
because no matter the angle $\theta$, the functions are always supported on a strip of width at most $2r_\Omega$.

In what follows, we abbreviate $\delta\lambda = \delta\lambda(\gamma_{s,\theta}(t))$, $\mu_i = \mu_i(\gamma_{s,\theta}(t))$ and $B\mu_i = B\mu_i(\gamma_{s,\theta}(t), \theta)$ for convenience.
Given $\mu_i$ are non-negative, $B\mu_i \geq 0$ anywhere and we can use \eqref{eq:expDiff}, Cauchy-Schwarz and \eqref{eq:beamBound} to estimate
\begin{align*}
    \hspace{-2em} I_1^2 &\leq \int_{0}^{2\pi} \int_\R \left( \int_\R \left| \delta \lambda \left( e^{-B\mu_1} - e^{-B\mu_2} \right) \right| dt \right)^2 \! ds \, d\theta \\ %
    &\leq \int_{0}^{2\pi} \int_\R \left( \int_{-r_\Omega}^{r_\Omega} \big|\delta \lambda \, B(\mu_1 - \mu_2)\big| dt \right)^2 ds \, d\theta \\ %
    &\leq \int_{0}^{2\pi} \! \int_{\R} 2r_\Omega \! \int_{-r_\Omega}^{r_\Omega} \big| \delta \lambda \big|^2 dt \int_{-r_\Omega}^{r_\Omega} \big| \mu_2 - \mu_1 \big|^2 \! dt \, ds \, d\theta \\ %
    &= 4\pi r_\Omega \| \delta \lambda \|_2^2 \, \| \mu_2 - \mu_1 \|_2^2.
\end{align*}

Similarly we can use Cauchy-Schwarz, \eqref{eq:expDiff} and \eqref{eq:beamBound} to estimate
\begin{align*}
    I_2^2 &\leq \int_{0}^{2\pi} \! \! \int_\R \! \bigg( \! \int_{-r_\Omega}^{r_\Omega} \bigg| B\delta \mu %
     \left( \lambda_2 e^{-B\mu_2} - \lambda_1 e^{-B\mu_1} \right) \bigg| dt \bigg)^2 \! ds \, d\theta \\ %
    \leq &\int_{0}^{2\pi} \! \! \int_{\R} \int_{-r_\Omega}^{r_\Omega} \! \left| B\delta \mu \right|^2 dt %
    \int_{-r_\Omega}^{r_\Omega} \! \left| \lambda_2 e^{-B\mu_2} - \lambda_1 e^{-B\mu_1} \right|^2 \! dt \, ds \, d\theta \\ %
    \leq &\int_{0}^{2\pi} \int_{\R} 2r_\Omega \int_{-r_\Omega}^{r_\Omega} \left| \delta \mu \right|^2 dt \ \times \\ %
    &\int_{-r_\Omega}^{r_\Omega} \big| \lambda_2 \big( e^{-B\mu_2} - e^{-B\mu_1} \big) \big|^2 %
    + \big| e^{-B\mu_1} \big( \lambda_2 - \lambda_1 \big) \big|^2 dt \, ds \, d\theta.
\end{align*}
Then apply again equations \eqref{eq:expDiff}, \eqref{eq:beamBound} and the boundedness of $\lambda_2$ to estimate
\begin{align*}
    & \int_{-r_\Omega}^{r_\Omega} \big| \lambda_2 \big( e^{-B\mu_2} - e^{-B\mu_1} \big) \big|^2 + \big| e^{-B\mu_1} \big( \lambda_2 - \lambda_1 \big) \big|^2 dt \\
    \leq & \int_{-r_\Omega}^{r_\Omega} \| \lambda_2 \|_\infty^2 \big| B\mu_2 - B\mu_1 \Big|^2 + \big| \lambda_2 - \lambda_1 \big|^2 dt \\ %
    \leq & \, 2 r_\Omega \| \lambda_2 \|_\infty^2 \! \int_{-r_\Omega}^{r_\Omega} \big| \mu_2 - \mu_1 \big|^2 dt + \int_{-r_\Omega}^{r_\Omega} \big| \lambda_2 - \lambda_1 \big|^2 dt.
\end{align*}
Therefore
\begin{align*}
    I_2^2 &\leq 2 r_\Omega \! \int_0^{2\pi} \! \! \| \delta \mu \|_2^2 \left( 2 r_\Omega \| \lambda_2 \|_\infty^2 \| \mu_2 - \mu_1 \|_2^2 + \| \lambda_2 - \lambda_1 \|_2^2 \right) d\theta \\
    &= 4\pi r_\Omega \| \delta \mu \|_2^2 \left( 2 r_\Omega \| \lambda_2 \|_\infty^2 \| \mu_2 - \mu_1 \|_2^2 + \| \lambda_2 - \lambda_1 \|_2^2 \right) \\ %
    &\leq 4\pi r_\Omega \| \delta \mu \|_2^2 \left( \sqrt{2 r_\Omega} \| \lambda_2 \|_\infty \| \mu_2 - \mu_1 \|_2 + \| \lambda_2 - \lambda_1 \|_2 \right)^2.
\end{align*}
\end{proof}

Although we formulate the model and the neural operator architectures in the continuous setting, i.e. using functions and operators, discretization is required in practical computations. The forward model is discretized as $\lambda, \mu \in \R^{\Npx}$ and 
\begin{align}
    y_n &= \sum_{m=1}^{\Npx} \lambda[m] r_{m,n} \exp \left( - c_{m,n} \sum_{k=1}^{\Npx} d_{m,n,k} \mu[k] \right). \label{eq:d_fwd_model}
\end{align}
Here the values $0 \leq r_{m,n} \leq 1$ give the probability that a photon emitted randomly from pixel $m$ reaches the collimated detector at angle $\theta_n$ and offset $s_n$; $c_{m,n} \geq 1$ are correction factors due to the vertical opening angle since the detector catches far-away emissions further above and below the central imaging plane and those rays have to travel diagonally through the attenuating medium; and $d_{m,n,k}$ are the distances a ray from $m$ to $n$ covers within pixel $k$ on the plane. Values for $r, c$ and $d$ can be precomputed, and the emission and attenuation maps $\lambda, \mu$ are rotated by the given angle $\theta_n$ when each projection is computed. The Jacobians wrt. discrete $\lambda$ and $\mu$ would follow in similar fashion from \cref{eq:frechet_partial}.

The forward model in \cref{eq:fwd_model} leads to the nonlinear least squares problem which we enhance with constraints $\mcC$ and regularization terms $P_1$ and $P_2$ \cite{engl1996regularization}:
\begin{equation} \label{eq:lma_objective}
    \argmin_{(\lambda, \mu) \in \mcC} \frac{1}{2} \big\| \F(\lambda, \mu) - y^\delta \big\|_2^2 + \alpha_1 \big\| P_1 \lambda \big\|_2^2 + \alpha_2 \big\| P_2 \mu \big\|_2^2.
\end{equation}
We use the same priors as in \cite{backholm2020simultaneous}, namely the combination of linear and box bounds (e.g. low attenuation but highly emitting materials are not allowed, since those do not naturally appear in any part of the fuel assembly) and projections $P_1$ and $P_2$ which penalize emission and attenuation values which are outside the approximate potential locations of the fuel rods given the known or predetermined fuel geometry. However, we do not assume that any rod location is filled a priori. The regularization parameters $\alpha_1, \alpha_2 > 0$ balance the strength of the geometry priors.

\subsection{Levenberg-Marquardt Algorithm} \label{ssec:LMA}
The LM trust region algorithm \cite{kelley1999iterative} is a well known method for minimizing nonlinear least-squares problems. Here we give a brief explanation following the earlier notation of the continuous setting.

At $k$'th iterate we can stack the unknown emission and attenuation by $u_k = (\lambda_k, \mu_k)$, the Forward operator and regularization operators by $\Ft (u) = \big( \F(\lambda, \mu), \alpha_1 P_1(\lambda), \alpha_2 P_2 (\mu) \big)$, the data terms by $\tilde{y} = (y^\delta, 0, 0)$ and denote the derivative of $\Ft$ at $u_k$ by $D\Ft (u_k) = \Ft'_k$ (which now contains both the data mismatch and regularization terms). We linearize the problem at $u_k$ and recover the update step $s_k$ from
\begin{align}
    \label{eq:linearized_prob}
    -\big( \Ft (u_k) - \tilde{y} \big) =: -r_k \approx \Ft'_k ( u_{k+1} - u_k ) = \Ft'_k s_k.
\end{align}
In consequence, for every iterate $k$, the update minimizes an objective function
\begin{align} \label{eq:m-function}
    m_k(s) &= \frac{1}{2} \| r_k \|^2 + \langle \Ft'_k s, r_k \rangle + \frac{1}{2} \| \Ft'_k s \|^2,
\end{align}
which we assume to provide a good approximation inside \emph{a trust region} $\| s \| \leq \Delta_k$. 
The linearized problem \eqref{eq:linearized_prob} inherits the regularization from the projections $\alpha_1 P_1, \alpha_2 P_2$, which can have a non-trivial null space and do not necessarily stabilize $F'$ for arbitrary search direction $s$. Therefore, instead of minimizing \eqref{eq:m-function} with a constraint on $s$, the objective function is further regularized using a Tikhonov-type penalty with a step-dependent weight (called the LM parameter) $\beta_k \geq 0$. Hence, the update step is
\begin{align} \label{eq:LMA-step}
   s_k &= \argmin_{s \in \tilde{\mcC}} \left\| \Ft'_k s + \big( \Ft (u_k) - \tilde{y} \big) \right\|_2^2 + \beta_k \big\| s \big\|_2^2 \\
    &= \argmin_{s \in \tilde{\mcC}}  \left\| \begin{bmatrix}
        \Ft'_k \\
        \sqrt{\beta_k} I
    \end{bmatrix} s + \begin{bmatrix}
         r_k \\
         0
    \end{bmatrix} \right\|_2^2. \nonumber
\end{align}
Here the constrained set $\tilde{\mcC}$ is set up such that $u_{k+1} = u_k + s_k \in \mcC$ is still satisfied. 
For every iterate a line search for good descent step is performed and the parameter $\beta_k$ is tuned accordingly \cite{kelley1999iterative}. For $\beta_k = 0$ it corresponds to the Gauss-Newton method. This constrained optimization task is done using the scaled gradient projection \cite{bonettini2008scaled}. More details on the formulation can be found in \cite{heikkinen2024model} and from now on we simply write
\begin{align*}
    \LMA(u_k) = s_k
\end{align*}
for the update step proposed by the deterministic algorithm. Using the current iterate $u_k$ and corresponding $s_k$ we want to learn a better update step $\tilde{s}_k$ to speed up the iteration. The LM method can be considered a regularization scheme, but its convergence (as noise level decreases) for general ill-posed inverse problems is not automatic, see \cite{hanke2010regularizing} and references therein. 
Here we only use the standard trust region approach and do not consider this validation. 

\section{Acceleration with Deep Gauss-Newton} \label{sec:Deep Gauss Newton}

The Deep Gauss-Newton method, introduced in \cite{hauptmann2018model, mozumder2021model} and inspired by \cite{andrychowicz2016learning}, uses a learned update function $\G_{\theta_k}$ at each iteration to improve the traditional gradient or Gauss-Newton step coming from the minimization of model-based term such as data mismatch $f(u_k) = \| \F(u_k) - y^\delta \|^2$. Then instead of updating $u_{k+1} = u_k + \xi \nabla f(u_k)$ at every iterate $k$, we do $u_{k+1} = \G_{\theta_k} (u_k, \nabla f(u_k))$. A separate network is needed for every iterate, but the architecture is usually kept fixed.

The purpose of this approach is to achieve high-quality reconstruction with fewer iterations. In our case we compute the proposed step $s_{k}$ with a deterministic algorithm, such as LM in Section~\ref{ssec:LMA}. Both the current point and proposed step are fed as input $a_k = (u_k, s_k)$ to a deep learning algorithm of choice. Output of the deep learning algorithm should be a better update than $s_k$ since it is trained to produce an output that approximates the ground truth of the sample $u^\dagger$. The new iterate $u_{k+1}$ is then fed to the deterministic algorithm to get $s_{k+1}$ to obtain new input $a_{k+1} = (u_{k+1}, s_{k+1})$ and so on. An overview of the training process is seen in Algorithm~\ref{alg:LMA_deep_algo}.
\begin{algorithm}[!ht]
  \caption{$\LMA$ + Deep Gauss-Newton}\label{alg:LMA_deep_algo}
\medskip
Given initial guesses and ground truths $\big( u_{(n)}; u^\dagger_{(n)} \big)_{n=1}^\Nsmp$ in the training data. \\
$k = 0$\\
Iterate \textbf{while} ($k < k_{\text{max}}$):
  \begin{itemize}
    \item Compute traditional update steps 
    $$s_{(n)} = \LMA(u_{(n)}), \ \textbf{for} \ n = 1, \dots, \Nsmp$$
    with the deterministic algorithm and initialize the training data: $\big\lbrace (u_{(n)}, s_{(n)}; u^{\dagger}_{(n)}) \big\rbrace_{n=1}^\Nsmp$.

   \item For $n = 1, \dots, \Nsmp$, prepare random rotation matrices $R_{\phi_n}$ for perturbing the inputs and ground truths. The rotation angles $\phi_n$ follow uniform distribution on the interval $[-5^\circ, 5^\circ]$.

    \item \textbf{train} $\G_{\theta_k}: (u,s) \mapsto \tilde{u}$ by minimizing the Mean Squared Error loss:
    \begin{equation*}
        \! \sum_{n=1}^\Nsmp \! \tfrac{1}{\Nsmp} \left\| \G_\theta(R_{\phi_n} u_{(n)}, R_{\phi_n} s_{(n)}) - R_{\phi_n} u^\dagger_{(n)} \right\|_2^2. 
    \end{equation*}
    \item Obtain parameters $\theta_k$ for the $k$th network.

    \item Evaluate test data and update training data (without random rotations $R_\phi$): 
    $$ u_{(n)} = \G_{\theta_k}( u_{(n)}, s_{(n)} ) \ \textbf{for} \ n=1, \dots , \Nsmp. $$
    \item Next iteration: $k = k + 1$.
  \end{itemize}
\end{algorithm}

The network steers the reconstruction toward the ground truth, as these values are used as targets during training. While this may override the update step of the standard LM method, it can also help correct modeling errors introduced, for example, by regularization. However, we note that the training data is simulated using the same approximate forward operator and lacks the effect of scattering, for example. By design, any priors and regularization only affect the training data. We would like the method to be agnostic to the fuel geometry, for example, and for it to correct modeling errors, rather than possibly introduce new ones by the user.

The deep learning network, defined as $\G_{\theta_k}: (u, s) \mapsto \tilde{u}$ can be based on different architectural choices. The different neural network or operator architectures applied in this study are CNNs, FNOs and WNOs, which will be explained in detail in \cref{ssec:CNN,,ssec:FNO,ssec:WNO}, respectively.

The deep Gauss-Newton partly overlaps with the related and popular Plug-and-Play (PnP) framework \cite{venkatakrishnan2013plug, kamilov2023plug}, which is inspired by forward-backward splitting and proximal gradient algorithms which minimize problems of form $f(u) + g(u)$ by iterating $u_{k+1} \gets \prox_{g}(u_k - \nabla f(u_k))$. If $g$ is a (nonsmooth) regularizer, its proximal operator looks like a denoiser and replacing it with a learned denoiser or operator leads to the PnP formulation $u_{k+1} \gets \G_\theta(u_k - \nabla f(u_k))$. The PnP formulation more explicitly replaces the regularization term $g$ with a learned component, whereas the deep Gauss-Newton can include the regularizer in the gradient $s_k = \nabla f(u_k) + \nabla g(u_k)$ as long as it is smooth. Unlike our method, the convergence of PnP schemes are often based on fixed point iteration theory and building nonexpansive architectures \cite{ryu2019plug, pesquet2021learning, bredies2024learning}.

Let us also mention other related methods, such as Learning to Optimize (L2O) \cite{chen2022learning} and algorithm unrolling \cite{monga2019algorithm} and in particular, their safeguarded variants \cite{heaton2023safeguarded, premont2022simple, fahy2024greedy} which we will closely touch upon in the next section.

\subsection{Convergence} \label{ssec:convergence}

In general, we can no longer guarantee convergence of the Deep Gauss–Newton iterates. However, by \cref{lemma:F'isLipschitz}, the Fréchet derivatives (and hence the Jacobians) are well behaved, and the LM-algorithm retains certain provable convergence properties even with approximate update steps \cite{kelley1999iterative, nocedal2006numerical}. In particular, the following statement holds.

\begin{theorem}[Thm. 10.3, \cite{nocedal2006numerical}] \label{thm:10.3}
Let $\eta \in (0, \tfrac{1}{4})$ in Algorithm 4.1 in \cite[Chapter 4]{nocedal2006numerical}, and suppose that the level set $\mcL = \lbrace u \, | \, f(u) \leq f(u_0) \rbrace$ is bounded and that the residual functions $r_j(\cdot), j = 1,\dots, m$ are Lipschitz continuously differentiable in a neighborhood of $\mcL$. Assume that for each $k$, the approximate solution $s_k$ of \cref{eq:LMA-step} satisfies the inequality
\begin{equation} \label{eq:ConvInEq}
    m_k(0) - m_k(s_k) \geq c_1 \big\| \Ft'^* [r_k] \big\| \min \left\lbrace \Delta_k, \frac{\| \Ft'^* [r_k] \| }{ \| \Ft'^* \Ft' \|} \right\rbrace,
\end{equation}
for some constant $c_1 > 0$, and in addition $\| s_k \| \leq \gamma \Delta_k$ for some constant $\gamma \geq 1$. We then have that
\begin{equation}
    \lim_{k \to \infty} \Ft'^* [r_k] = 0.
\end{equation}
\end{theorem}

Hence, we propose the following convergence condition which guarantees that the accelerated iterates converge to a critical point.

\begin{proposition} \label{prop:conv_check}
    Assume $\LMA: u_k \mapsto s_k$ computes the Levenberg-Marquardt trust region algorithm update steps to minimize the objective functional $f: X \to \dot \R$, with the LM parameter $\beta_k$ and trust region update scheme similar to \cite[Algorithm 4.1]{nocedal2006numerical} which follows the assumptions of \cref{thm:10.3}.
    
    For an operator $\G: (u,s) \mapsto u + L(s) \odot g(u,s)$ with bounded and linear $L: X \to X$ and bounded operator $g: X \times X \to X$, let
    \begin{align}
        &s_k = \LMA(u_k), \quad \tilde{u}_k = \G (u_k, s_k), \nonumber \\
        &\tilde{s}_k = \G(u_k, s_k) - u_k \quad \text{and} \nonumber \\
        &\rho_k(s) = \frac{f(u_k) - f(u_k + s)}{m_k(0) - m_k(s)}. \label{eq:rho_agreement}
    \end{align}
    Then for fixed $\kappa > 0$ the iterates
    \begin{align} \label{eq:uk+1_condition}
    u_{k+1} &= \begin{cases}
        \tilde{u}_k,  \ \text{if} & \hspace{-1.5em} m_k(\tilde{s_k}) \leq m_k(s_k) \ \text{and} \ \rho_k(\tilde{s}_k) > \kappa \\
        u_k + s_k, &\text{otherwise},
    \end{cases}
\end{align}
converge to a critical point of $f$.
\end{proposition}
\begin{proof}
The condition in \cref{eq:uk+1_condition} guarantees that the updates to $u_k$ always satisfy the inequality~(\ref{eq:ConvInEq}), and there is a meaningful decrease in $f$.

The LM-trust region algorithm \cite{kelley1999iterative} is based on sufficient agreement between $f$ and $m_k$ such that
\begin{align*}
    &\frac{f(u_k) - f(u_k + s_k)}{m_k(0) - m_k(s_k)} = \rho_k(s_k) > \eta, \\
    \Rightarrow \ & f(u_k) - f(u_k + s_k) > \eta \left( m_k(0) - m_k(s_k) \right)
\end{align*}
for the step $s_k$ and parameter $\beta_k$ to be accepted in the first place. And if $\rho_k(\tilde{s}) > \kappa$, then
\begin{align*}
    f(u_k) - f(u_{k} + \tilde{s}_k) \geq \kappa \left( m_k(0) - m_k(\tilde{s}_k) \right).
\end{align*}
Therefore
\begin{align*}
    f(u_k) - f(u_{k+1}) \geq \min \{ \kappa, \eta \} \left( m_k(0) - m_k(s_k) \right).
\end{align*}
This means that the objective function values should decrease monotonically and not stall, unless $s_k \to 0$.

Given the assumptions and \cref{lemma:F'isLipschitz}, we can apply \cref{thm:10.3} which implies that the gradients $\nabla f(u_k) = \Ft'^*_{k} r_k \to 0$, where we denote the derivative ${\Ft'}_k = D\Ft (u_k)$ and the residual $r_k = \Ft(u_k) - \tilde{y}$. 
By equation~\eqref{eq:LMA-step}
\begin{align*}
   s_k &= \argmin_{s \in \tilde{\mcC}} \left\| \Ft'_k s + r_k \right\|_2^2 + \beta_k \big\| s \big\|_2^2 \\
    &= \argmin_{s \in \tilde{\mcC}} \| \Ft'_k s \|_2^2 + \underbrace{2 \langle s, {\Ft'^*_{k}} r_k  \rangle}_{\to 0} + \| r_k \|_2^2 + \beta_k \| s \|_2^2,
\end{align*}
which implies that the LM updates $s_k$, evaluated at $u_k$ tend to zero as $k \to \infty$. Since $\G(u_k, s_k) = u_k + L(s_k) \odot g(u_k, s_k)$, where $\odot$ denotes the elementwise product and $L$ and $g$ are bounded operators, then $L(s_k) \odot g(u_k, s_k) \to 0$ as well and $u_{k+1}$ converges to a critical point where $\nabla f(u^*) = 0$.
\end{proof}

We note that the proposition only assumes that the operator $g$ (i.e. the neural network or operator architecture) preserves boundedness: $\| g(x) \| \leq C \| x \|$ for some constant $C > 0$. This is true for most architectures, although the magnitude of the constant could be massive. Yet, if $\G_\theta$ does not map the iterates towards a sufficient decrease in $m_k$ and $f$, the accelerator is skipped completely and the original LM iterate is used as a safeguard to retain convergence. In such a case we may choose to pass the traditional iterate $u_{k+1} = u_k + s_k$ to either the next network $\G_{\theta_{k+1}}$ or retry the previous one $\G_{\theta_k}$. In practice, as $\G$ is trained using ground truth data $u^\dagger$ while the LM-algorithm drives towards the regularized solution $u^*$ (with fixed $\alpha$), a mismatch in objectives arises. In consequence, one should terminate the iteration early according to e.g. the discrepancy principle.

In our implementation, we saw better results by decreasing the regularization parameter values by a factor of 5 after every iterate: $\alpha^{(k)} = (\alpha_1, \alpha_2) / 5^k$ and assumed the accelerator $\G_{\theta_k}$ would act as an implicit regularizer on the iterates. We did halt the decrease for some $k$ onward to keep $f$ fixed, as the convergence of this kind of \emph{fixed-point continuation method} does not follow automatically for the LM-algorithm and hence \cref{thm:10.3} and \cref{prop:conv_check} are only applicable for fixed $\alpha$. Extending the convergence of LM to changing values of $\alpha_k$ along the lines of \cite{fest2022fixed, Goldfarb2011, Hale2008} or other means of tailoring $u^*$ towards $u^\dagger$ would be interesting in the future, but for now our computations do not fully match the theorems. The whole scheme is summarized in \cref{alg:check_convergence}. In all our tests, threshold values of $\kappa = 10^{-4}$ and $\eta = 10^{-3}$ were used. In practice the safeguard condition \eqref{eq:uk+1_condition} mostly failed because the objective function value did not decrease and the ratio $\rho$ becomes negative, no matter the value of $\kappa$.

\begin{algorithm}[!ht]
  \caption{Accelerated iteration scheme with checks.}\label{alg:check_convergence}
\medskip
Given initial input $u_0$, trained neural networks or operators $\G_{\theta_k}$ for $k = 0$ and discrepancy parameter $\kappa > 0$.\\
Iterate \textbf{while} ($k < k_{\text{max}}$):
  \begin{itemize}
    \item Compute traditional update step 
    $$s_{k} = \LMA(u_k)$$
    with the deterministic algorithm, given regularization parameters $\alpha_{k,1}, \alpha_{k,2}$.
    \item Compute model-based update 
    \begin{align*}
        \tilde{u} &= \G_{\theta_k}(u_k, s_k), \\
        \tilde{s} &= \tilde{u} - u_{k}.
    \end{align*}

    \item Compute $\rho(\tilde{s})$ as in \cref{eq:rho_agreement} and $m_k(\cdot)$ as in \cref{eq:m-function}.\\
    \textbf{if} $\rho(\tilde{s}) > \kappa$ \textbf{and} $m_k(\tilde{s}) \leq m_k(s_k)$, \textbf{then}
    $$ u_{k+1} = \tilde{u},$$
    \textbf{else}, skip the accelerator
    $$ u_{k+1} = u_{k} + s_k.$$

    \item Next iteration: $k = k + 1$.
  \end{itemize}
\end{algorithm}

We can think of the pointwise multiplication with $L(s_k)$ as a mask, which limits the potential updates of the neural network or operator to points updated by the deterministic algorithm. Alternatively, since $s_k$ is a descent direction, the diagonal matrix $D_k = \diag(g_\theta(a_k))$ could be considered a learned preconditioner or adaptive method \cite{koolstra2022learning, liao2022learning, li2023learning, fahy2024greedy}. In particular \cite{fahy2024greedy} considers general learned matrix preconditioners which are not necessarily symmetric positive definite, similar to our case. Compared to many safeguarded L2O methods such as \cite{heaton2023safeguarded, fahy2024greedy, premont2022simple}, our method requires monotonicity in objective function values due to \eqref{eq:rho_agreement}. However, the trust-region framework should offer adaptivity at iterates where the Jacobian is especially ill-conditioned.

\subsection{Common structure} \label{ssec:structure}

In this study we compare recent neural operator architectures in a regime where training data are scarce and synthetic. Accordingly, all models were constructed with comparable architectures and trained and evaluated on identical datasets. The input to each model was always 4 images of size $165 \times 165$ pixels, containing the concatenated emission and attenuation maps and the proposed iterate: $a = (\lambda, \mu, s_\lambda, s_\mu)$. The model output was always 2 images (emission and attenuation) of size $165 \times 165$. However, since the forward model in \eqref{eq:d_fwd_model} only considers the central disk within the images, all values outside the disk were always masked to zero in order to have reliable error and loss metrics.

By design, the deterministic algorithm solves a constrained optimization problem and hence the values of the proposed LM iterate $u_k + s_k$ are always within predetermined physical bounds (i.e. $[0, \lambda_{\max}]$ and $[0, \mu_{\max}]$) for emission and attenuation. When experimenting with training the models, we found it beneficial to normalize the inputs using the respective upper bounds. The model output was always rescaled back, but nothing in the structure of the networks explicitly forces the values to be within those bounds. The initial guess for all iterations was emission and attenuation of water everywhere ($\lambda_0 \equiv 0, \mu_0 \equiv 0.0085$).

\subsection{Training setup} \label{ssec:training data}
All neural networks and operators were trained with the same simulated data. We initialized 1200 samples of emission and attenuation data where the rods are in a rectangular $9 \times 9$ pattern (mimicking the real $9 \times 9$ geometry \cite{virta2024gamma}). The exact rod locations were slightly perturbed by a random amount to simulate the deformations seen in the real fuel due to the extreme environment inside an operating nuclear reactor. To avoid inverse crime, the emission and attenuation images were generated at twice the spatial resolution and downsampled during the forward projection step.

Each individual fuel rod has a random emission value chosen from a uniform distribution in the interval $[6.5 \cdot 10^5, 7 \cdot 10^5]$ and an attenuation value in $[0.12, 0.14]$ ($1/$mm). Water always has an attenuation value of $0.0085$ and no emission.

Of the 1200 'standard' samples, 200 samples have all 81 rods present, 500 samples have 1 to 6 random rods missing (replaced by water) and 500 samples have 5 to 6 random rods replaced by non-emitting rods (simulating fresh fuel). The projection data were corrupted with 2\% Gaussian noise. From these 1200 samples, 1000 were used for training and 200 for validation.

During the experiments, we added 200 samples of 'medium' difficulty to the training data. It contains a mix of 1 to 20 missing and replaced rods, where the attenuation value of a replaced rod (no emission) and the emission value of a present rod can be as low as 70\% of the maximum value. However, the attenuation value of an emitting rod was always within the original interval to remain consistent with the physical bounds expected of the rods.

Finally, since the real data are not oriented exactly in the image domain (although the data preprocessing tries to do so), we randomly rotate both input and ground truth emission and attenuation image pairs (by up to $\pm 5^\circ$) during training of the networks. This way the final network should be more tolerant to imperfectly aligned input images.

All models were built using PyTorch~\cite{PyTorch}, had roughly the same number of parameters, and an AdamW optimizer~\cite{loshchilov2018decoupled} with weight decay i.e. Tikhonov regularization, was employed for the training. The applied hyperparameter values and choices for each neural network architecture are presented in \cref{tab:hyperparameters}.

\begin{table*}[t]
    \renewcommand{\arraystretch}{1.2}
    \centering
    \caption{Hyperparameter values used in training different neural networks. All 5 blocks of the same architecture use the same hyperparameters, except the regularization parameters, which vary with the block number.}
    \begin{tabular}{c|ccc}
         Hyperparameter & CNN & FNO  & WNO \\ \hline
         Learning rate & $8 \times 10^{-3}$ & $3 \times 10^{-3}$  & $8 \times 10^{-3}$ \\
         Weight decay & $1 \times 10^{-4}$ & $1 \times 10^{-3}$  & $1 \times 10^{-4}$ \\
         Scheduler & CosineAnnealingLR & CosineAnnealingLR  & CosineAnnealingLR \\
         Optimizer & AdamW & AdamW  & AdamW \\
         Epochs & 6 & 6  & 6 \\
         Batch size & 5 & 5  & 5 \\
         Activation function $\sigma$ & ReLU & GELU & GELU \\
         Linear operator $L$ & 2D convolution & local FNO layer & identity \\
         Regularization $\alpha^{(k)}$ & ($2 \times 10^{-5}$, 3)$/{5^{k}}$ & ($2 \times 10^{-5}$, 3)$/{5^{k}}$ & ($2 \times 10^{-5}$, 3)$/{5^{k}}$ \\
    \hline
    \end{tabular}
    \label{tab:hyperparameters}
\end{table*}

\subsection{Convolutional Neural Network} \label{ssec:CNN}

CNNs are deep learning models that automatically extract hierarchical features from structured data, especially images~\cite{leCunConv}. They effectively capture spatial dependencies while significantly reducing the number of trainable parameters compared to fully connected networks~\cite{Dive_into_DL}. Through layers of convolution, nonlinear activation, and pooling, CNNs progressively translate low-level patterns into high-level semantic representations.

A single convolutional operation~\cite{Dive_into_DL} for a two-dimensional matrix input $\mathbf{v}$ and output $\mathbf{u}$ is given as
\begin{equation}
  \mathbf{u}[i] = \sigma \left( b + (w * \mathbf{v})[i] \right), \quad i = 1,2,\ldots, N
\end{equation}
where $\mathbf{v}[i] := v(x_i)$ are the pixels of discretized sample at $x_i$, $\sigma$ is the ReLU~\cite{Deep_learning_Python_book} activation function following the two-dimensional convolution $*$ with $5 \times 5$ convolution kernels. The network's trainable parameters are denoted by $\theta$, including kernel parameters $w$ and bias vectors $b$. The linear operator $L$ from \cref{prop:conv_check} was also set to $5 \times 5$ convolution layer without any bias or activation function. In total, the CNN model had 28 104 trainable parameters. All the subsequent CNNs for other iterates have their own trainable parameters $\theta_k$.

The neural network approach employs an encoder–decoder architecture~\cite{encoder_decoder_architechture,Dive_into_DL}, a paradigm that converts input data into a compact latent representation and reconstructs the target output from it. The encoder captures the most relevant features through successive convolutional layers, while the decoder restores spatial resolution or produces structured predictions from the latent representation. The simplified structure of the CNN used is shown in \cref{fig:CNN_arc}. The number on each block indicates the number of data channels which first increase as the spatial resolution decreases (encoding phase) and after combining the latent forms, this process is reversed (decoding phase).
\begin{figure}[htb]
  \centering
\scalebox{0.9}{
\begin{tikzpicture}
  \node (xk) [left] {\large $u$};
  \node (delta_xk) [below of=xk, node distance=3.5cm] {\large $s$};
  
  \node (block0) [block, below=0.005cm of xk] {2};
  \node (block1) [block, right=0.33cm of block0] {8};
  \node (block2) [block, right=0.33cm of block1] {16};
  \node (block3) [block, right=0.33cm of block2] {32};
  
  \node (block01) [block, above=0.005cm of delta_xk] {2};
  \node (block4) [block, right=0.33cm of block01] {8};
  \node (block5) [block, right=0.33cm of block4] {16};
  \node (block6) [block, right=0.33cm of block5] {32};
  
  \node (sum1) [sum, below of=block3, node distance=1.25cm] {\textbf{+}};
  
  \node (block7) [block, above right=0.3cm and 0.3cm of sum1] {32};
  \node (block8) [block, right=0.33cm of block7] {16};
  \node (block9) [block, right=0.33cm of block8] {8};
  \node (block10) [block, right=0.33cm of block9] {2};

  \node (prod) [sum, below=0.45cm of block10] {\large$\odot$};
  \node (extra) [block, below=0.6cm of prod] {2};
  
  \node (sum2) [sum, right=0.35cm of prod] {\textbf{+}};
  \node (out) [below =0.65cm of sum2] {\large $\tilde{u}$};

  \draw [scalar_arrow] (delta_xk) -- ++(4.5,0) |- (extra);
  \draw [basic_arrow] (extra) -- (prod);
  \draw [basic_arrow] (prod) -- (sum2);
  
  \draw [arrow] (block0) -- (block1);
  \draw [arrow] (block1) -- (block2);
  \draw [arrow] (block2) -- (block3);
  \draw [basic_arrow] (block3) -- (sum1);
  
  \draw [arrow] (block01) -- (block4);
  \draw [arrow] (block4) -- (block5);
  \draw [arrow] (block5) -- (block6);
  \draw [basic_arrow] (block6) -- (sum1);
  
  \draw [arrow] (sum1) -- (block7);
  \draw [arrow] (block7) -- (block8);
  \draw [arrow] (block8) -- (block9);
  \draw [scalar_arrow] (block9) -- (block10);
  \draw [basic_arrow] (block10) -- (prod);
  \draw [relu_arrow] (sum2) -- (out);
  
  \draw [skip] (xk.north) -- ++(0,0.3) -| (sum2.north);
  
  \begin{scope}[xshift=-8.5cm, yshift=-2.5cm]
      \draw [arrow] (8,-1.5) -- (8.5,-1.5) node[draw=none, right, color=black] { ReLU(conv2d)};
      \draw [scalar_arrow] (8,-2) -- (8.5,-2) node[draw=none, right, color=black] {conv2d};
      \draw [relu_arrow] (11,-1.5) -- (11.5,-1.5) node[draw=none, right, color=black] { ReLU};
      \draw [skip] (11,-2) -- (11.5,-2) node[draw=none, right, color=black] { skip connection};
  \end{scope}
  \end{tikzpicture}
  }
  \caption{Visualization of the convolutional neural network architecture. The dark solid arrows represents a convolutional layer followed by ReLU activation function, the light solid arrow represents a convolutional layer followed by an element-wise multiplication. The skip connection update (dark dashed arrow) is projected to the positive numbers by ReLU activation function (light dashed arrow).}\label{fig:CNN_arc}
\end{figure}

\subsection{Fourier Neural Operators} \label{ssec:FNO}
FNOs are recently introduced neural networks \cite{li2020fourier} which aim at approximating the point samples of infinite dimensional mappings, i.e. operators, in a resolution or mesh free way. They have shown great numerical accuracy in solving extremely difficult partial differential equations (PDEs) such as turbulent regime Navier-Stokes equations \cite{takamoto2022pdebench}.

An FNO is a parametrized mapping of the form
\begin{equation} \label{eq:FNO}
    \G_\theta = Q \circ F_L \circ \dots \circ F_1 \circ P: a(x) \mapsto u(x),
\end{equation}
where the key component are the FNO layers $F_l$, which are special convolution layers defined as
\begin{equation} \label{eq:FNOlayer}
    \hspace{-1em} F_l: v_l \mapsto \sigma \left( A(v_l) + \iF\Big( R \odot \mcF v_l \Big) \right) = v_{l+1},
\end{equation}
via the convolution theorem: convolution corresponds to pointwise multiplication $\odot$ on the Fourier domain. Here $R = R_{\theta,l} \in \C^{q}$ are the learnable convolution filter coefficients on the Fourier domain (alternatively we can think of $R$ as a diagonal matrix in $\C^{q \times q}$), $A = A_{\theta,l} \in \R^w$ is a linear operator (matrix) acting on the channels of the input $v_l$ and $\sigma$ is the nonlinear activation function.

The $P$ and $Q$ operators are known as lifting and projection operators, respectively. These are simple, parametrized multilayer perceptrons (MLPs) which increase and decrease the channel dimension of any input to a predefined width, whereas the FNO layers $F_l$ are able to operate in parallel on those channels. The general form of applied FNO architecture is shown in \cref{fig:FNO_arc}. Since $P, Q$ and each $A_l$ only operate along channels, independent of length of the input vector, the sampling resolution does not matter. And by computing the convolution on the Fourier domain, the (discrete) Fourier transform of fixed bandwidth $q$ and its inverse automatically upsample and downsample the input as needed.

\begin{figure}
    \centering
    \scalebox{0.9}{
    \begin{tikzpicture}
        \node[myrec={lightBlue}] (ll) at (0,-2) {Layer $\ell$};
        \node[mynode={mediumBlue}, right=4mm of ll] (proj) {$Q$};
        \node[minimum width=1cm, left=4mm of ll] (dots) {\Large $\hdots$};
        \node[myrec={lightBlue}, left=4mm of dots] (l1) {Layer 1};
        \node[mynode={mediumBlue}, left=4mm of l1] (lift) {$P$};
        \node[mynode={whiteBlue}, left=4mm of lift] (x) {$(u,s)$};
        \node[mynode={v-harmaa}, above=4mm of proj] (prod) {\large$\odot$};
        \node[myrec={mediumBlue}, left=30mm of prod] (Lop) {localFNO};
        \node[mynode={v-harmaa}, right=4mm of prod] (sum) {\textbf{+}};
        \node[mynode={whiteBlue}, right=4mm of proj] (v) {$\tilde{u}$};
        
        \draw[arrow, t-harmaa] (x.east) -- (lift.west);
        \draw[arrow, t-harmaa] (lift.east) -- (l1.west);
        \draw[arrow, t-harmaa] (l1.east) -- (dots.west);
        \draw[arrow, t-harmaa] (dots.east) -- (ll.west);
        \draw[arrow, t-harmaa] (ll.east) -- (proj.west);
        \draw[arrow, t-harmaa] (proj) -- (prod);
        \draw[arrow, t-harmaa] (x.north) -- ++(0.1,0.1) |- (Lop);
        \draw[arrow, t-harmaa] (Lop) -- (prod);
        \draw[arrow, t-harmaa] (prod) -- (sum);
        \draw[arrow, t-harmaa, dashed] (x.north) -- ++(-0.1,0.1) -- ++(0,1.5) -| (sum);
        \draw[arrow, t-harmaa] (sum) -- (v);
        
        \coordinate (r) at (-4.,-3);
        \draw[rounded corners, fill=v-harmaa, opacity=0.5] (r) rectangle +(6.5,-2);
        \draw[dashed] (r) -- (ll.south west);
        \draw[dashed] ($(r) + (6.5,0)$) -- (ll.south east);
        
        \node[mynode={whiteBlue}] (u) at (-3.5,-4.0) {$v$};
        \node[myrec={lightBlue}] (F) at (-1.0,-3.8) {$\iF \left( R \odot \mcF (v) \right) (x)$};
        \node[mynode={mediumBlue}] (W) at (-2.0,-4.6) {$Av(x)$};
        \node[mynode={darkBlue}] (plus) at (0.9,-4.6) {\textbf{+}};
        \node[mynode={darkBlue}] (sigma) at (2,-4.0) {$\sigma$};
        
        \draw[arrow, t-harmaa] (u.east) -- (F.west);
        \draw[arrow, t-harmaa] (F.east) -| (plus.north);
        \draw[arrow, t-harmaa] (u.east) -- (W.west);
        \draw[arrow, t-harmaa] (W) -- (plus);
        \draw[arrow, t-harmaa] (plus) -- (sigma);
        
        \node[anchor=center, mynode={white}, draw=black] at ($(r) + (3.5,0)$) {FNO layer};
    \end{tikzpicture}
    }
    \caption{Illustration of the FNO architecture with initial lifting layer $P$, $\ell$ FNO layers and final projection layer $Q$. Each FNO layer performs spatial convolutions on the Fourier domain at a fixed bandwidth and the linear operator $A$ mixes the spatial points channel-wise.}
    \label{fig:FNO_arc}
\end{figure}

It is worth noting that by construction, the convolutions favor low-pass structure and frequency components higher than $q$ can only carry from previous layer via $A_l$ or be introduced by the nonlinearity $\sigma$. While FNOs have seen a lot of activity in the PDE domain \cite{li2020fourier,takamoto2022pdebench}, their use in imaging tasks has been more limited, cf. \cite{liu2025enhancing,dai2023neural,kabri2023resolution}. Historically, the Fourier transform has been a key tool for solving various differential equations, whereas many image classes are characterized by local features and sharp edges and therefore better handled by local convolution filters, wavelets and CNNs, cf. \cite{gonzalez2018digital, mallat1999wavelet, zhao2024review}. This has motivated a variety of other neural operator architectures \cite{liu2025enhancing,tripura2023wavelet,raonic2023convolutional}.

We used the NeuralOperator library \cite{kovachki2021neural, kossaifi2024neural} to implement our FNO model. We chose GELU activation functions \cite{hendrycks2016gaussian}, 3 FNO layers, each with 12 hidden channels, $20 \times 20$ frequency bands Tucker-factorized to rank of at most 15\% and 8 lifting and projection channels, respectively. The linear operator $L$ from \cref{prop:conv_check} was set to a single local FNO layer \cite{liu2024neural} with no bias or activation, $5 \times 5$ convolution kernel and a linear skip connection. In total, the model had 30 532 tunable parameters.

The FNO had noticeably more trouble with the real data, so during training we added small amounts of ($0.1 \%$ relative) white Gaussian noise to every input image, which seemed to improve the results. We assume that by construction the FNO layers are challenged by the task of separating measurement noise and scattering from the high frequency structures in the images, since such errors are not present in the simulated data. An example can be seen in the FNO images of the ATRIUM-10 data in \cref{fig:atrium-results}.

\subsection{Wavelet Neural Operators} \label{ssec:WNO}
WNOs are a recently introduced \cite{tripura2023wavelet} family of neural operators, which have a general structure similar to FNOs (see \cref{eq:FNO}), but the Fourier-based convolution layers in \cref{eq:FNOlayer} are replaced with a learned wavelet-based integral operators or WNO layers of form
\begin{equation} \label{eq:WNOlayer}
    v_l(x) \mapsto \sigma \left( A(v_l) + \iW\Big( R ( \W v_l ) \Big) \right) (x) = v_{l+1}(x),
\end{equation}
where $R \in \R^{q \times q}$ is a learned linear operator, in general either a diagonal matrix or a convolution matrix and the size $q$ is given by the number of retained wavelet coefficients in $\W v_l$. By restricting the (discrete) wavelet transform only to the detail and approximation coefficients at some coarse resolution scale, the integral operator acts on an intermediate scale between the original pointwise samples in $v_l$ and the global sinusoids of the Fourier transform $\mcF v_l$.

Since the wavelet transform is not shift invariant and can not have a convolution theorem \cite{stone1998convolution}, the operation $\iW\big( R ( \W v_l ) \big)$ does not correspond to a convolution of $v_l$. However, as shown in \cite{lanthaler2025nonlocality}, as long as the underlying transform can represent the constant function, a universal approximation is possible.

Similarly, because the discrete wavelet transform is sensitive to shifting and uses dyadic scaling (where the next coarser resolution scale always contains half as many coefficients along each dimension), resolution invariance should only hold (even approximately) if the input resolutions are multiples of a power of two: $N_i = N 2^i$ and the depth of the wavelet transform is set adaptively to match $i$ such that the same wavelet scale is always used independent of the original input resolution. At least the WNO implementation \cite{tripura2023wavelet, WNO_code} used in this study is not adaptive and hence technically not resolution invariant. However, for this study it does not matter as we only consider fixed input resolution.

We chose GELU activation functions \cite{hendrycks2016gaussian}, 2 WNO layers, each with 12 hidden channels, 3 level Daubechies-3 discrete wavelet transform and $5 \times 5$ convolution kernels applied separately to the approximation coefficients and the three different directional detail coefficient subbands. The linear operator $L$ from \cref{prop:conv_check} was set to be the identity. In total, the model had 29 330 tunable parameters.

\section{Results} \label{sec:results}

In this section we briefly cover the resulting reconstructions from the CNN and different neural operators. We tested the network performance on simulated data, which were generated in a similar manner to the training data (see \cref{ssec:training data}), but which vary in difficulty. We expected some of the very unusual simulated data to indicate how well the neural network has generalized, but as the real data examples show, this is not always enough to guarantee robust output.

In \cref{fig:rel_errors} the relative $\ell^2$-errors are shown for each chosen iterate of simulated data. Given a nonzero reference emission $\lambda^\dagger$, the relative error of $\lambda$ is computed as $\tfrac{\| \lambda - \lambda^\dagger \|_{\ell^2}}{\| \lambda^\dagger \|_{\ell^2}}$ (similarly for attenuation $\mu^\dagger, \mu$). The reconstructions are also shown in \cref{fig:cnn_testdata,fig:fno_testdata,fig:wno_testdata}.
For comparison, the plots also include errors for running the LM-algorithm for 7 iterations without acceleration, and the resulting error after 15 iterations to indicate the quality of the traditional algorithm after a sufficiently long run.

The ground truth images and accepted iterates for two samples of two simulated datasets are shown in \cref{fig:cnn_testdata} for the CNN, in \cref{fig:fno_testdata} for the FNO and in \cref{fig:wno_testdata} for the WNO. In each figure, if the accelerated iterate $\tilde{u}_k$ did not satisfy the condition of \cref{eq:uk+1_condition}, the image is framed with red dashes. The figures show only the central section of the image with the fuel rods, but the relative $\ell^2$-errors in \cref{fig:rel_errors} are computed from the whole image.

Reporting iterations rather than actual computational time is somewhat biased against the traditional LM-algorithm, since every accelerated iteration requires computing one LM iteration and multiple objective function evaluations for the safeguard. However, these extra steps are roughly constant time operations whereas the LM iterate contains an inner iteration for the LM parameter and can sometimes take much longer to compute (ranging from 8 to 120 seconds). When considering the average wall-time computational cost of each iteration, 50-60\% of the computational time is spent on the LM iteration and 40-50\% is spent checking the sufficient decrease condition. Time spent loading and evaluating the learned model was negligible, always less than 1\% of computational time. Both the LM-step and model evaluations utilize GPU acceleration, but currently the safeguard check does not. The 15 traditional LM iterates were always slower to compute than 5 iterates with the safeguard check. These tests were carried out on a Dell Precision 3591 laptop with 32 GB of RAM, Intel Core Ultra 5 135H CPU and Nvidia RTX 500 Ada Generation GPU.

Although the training time for each model varies, the differences in total training time are less noticeable when most of the computational cost lies in computing the proposed LM iterate $s_{k+1}$ for the next training round. And since these updates depend on the output of the previous model $\G_{\theta_k}$, this costly step must be repeated for every model and iterate. On our small computational cluster the total training times were 6 hours 18 minutes for CNN, 4 hours 52 minutes for FNO and 5 hours 11 minutes for WNO. These times also include updating the test data, but that step could also be carried out afterwards.

\subsection{Simulated test data} \label{ssec:test data}

The simulated data is split into following categories.
\begin{itemize}
    \item 'Standard': 200 samples used for evaluation during training. Randomly picked from the 1200 samples (remaining 1000 used for training).
    \item 'Hard': 10 samples generated similar to the 'medium' difficulty but with more rods missing and replaced and larger variations in emission and attenuation values.
    \item 'Extreme': 4 handcrafted samples with very few rods in particularly difficult arrangements.
\end{itemize}
The simulated sinograms were corrupted with 2\% white Gaussian noise. But unlike the training data, the test data was not affected by additional noise between iterations or random rotations. The 'Hard' and especially 'Extreme' samples have more missing and partial rods than any real fuel assembly should, but this is on purpose to check the robustness of the model-based methods against outliers outside the training data.

\begin{figure}
    \centering
    \includegraphics[width=1.0\linewidth]{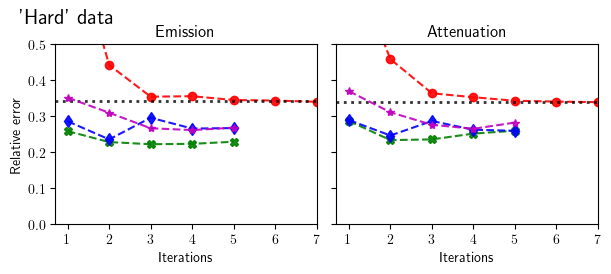}\\
    \includegraphics[width=1.0\linewidth]{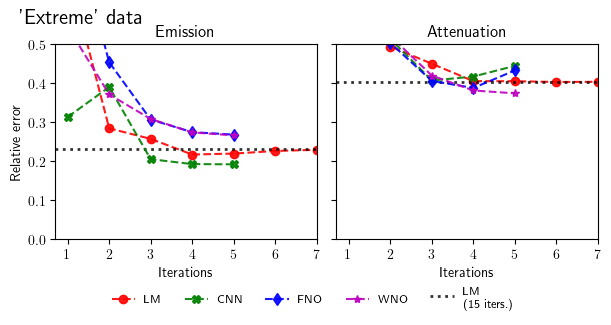}
    \vspace*{-2em}
    \caption{Relative $\ell^2$-errors of the emission and attenuation images for 5 iterates with different methods and simulated datasets. The errors of the LM-algorithm after 15 iterations are also shown in dark gray for comparison.}
    \label{fig:rel_errors}
\end{figure}

\begin{figure}
\setlength{\imSz}{0.16\columnwidth}
\setlength{\imSkip}{0.92\imSz}
    \centering
    \begin{tikzpicture}
    \begin{scope}
        \node[anchor=north] (im) at ($(-1.0\imSz,0)$) {\includegraphics[width=1.35\imSz]{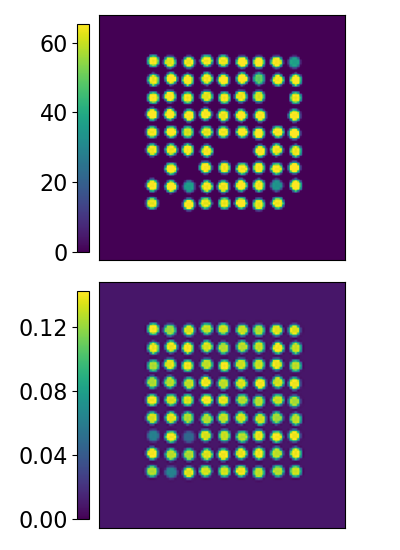}};
        \node[rotate=90] at ($(-1.8\imSz,-\imSz)$) {\small 'Hard' data};
        \node[above=-1mm of im, align=center, scale=.9] {Ground\\ truth};
        \foreach \iter [count=\x from 0] in {1,...,5}{
            \node[anchor=north] (out) at ($(\x*\imSkip, 0)$) {\includegraphics[width=\imSz]{images/cnn/simulated/decay2_cnn_model_s6_n_hard_10_samples-3__pass_-mode_iter\iter.png}};
            \node[above=-1mm of out, scale=.9] {Iter \iter};
        }
        \node[rotate=0] at ($(4.6*\imSkip,-0.5*\imSz)$) {$\lambda$};
        \node[rotate=0] at ($(4.6*\imSkip,-1.43*\imSz)$) {$\mu$};
    \end{scope}
    \begin{scope}[yshift=-2\imSz]
        \node[anchor=north] (im) at ($(-1.0\imSz,0)$) {\includegraphics[width=1.35\imSz]{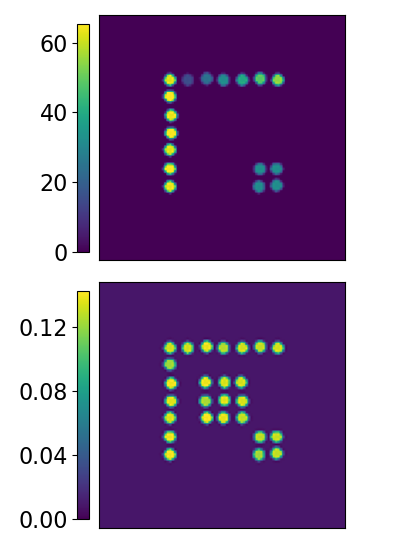}};
        \node[rotate=90] at ($(-1.8\imSz,-\imSz)$) {\small 'Extreme' data};
        \foreach \iter [count=\x from 0] in {1,...,5}{
            \node[anchor=north] (out) at ($(\x*\imSkip, 0)$) {\includegraphics[width=\imSz]{images/cnn/simulated/decay2_cnn_model_s6_n_crazy_data2-1__pass_-mode_iter\iter.png}};
        }
        \node[rotate=0] at ($(4.6*\imSkip,-0.5*\imSz)$) {$\lambda$};
        \node[rotate=0] at ($(4.6*\imSkip,-1.43*\imSz)$) {$\mu$};
    \end{scope}
    \end{tikzpicture}
    \caption{CNN iterates on one sample from 'Hard' and 'Extreme' test data sets. If the LM iterate is accepted in place of the model's proposition, the image is framed with red dashes.}
    \label{fig:cnn_testdata}
\end{figure}

\begin{figure}
\setlength{\imSz}{0.16\columnwidth}
\setlength{\imSkip}{0.92\imSz}
    \centering
    \begin{tikzpicture}
    \begin{scope}
        \node[anchor=north] (im) at ($(-1.0\imSz,0)$) {\includegraphics[width=1.35\imSz]{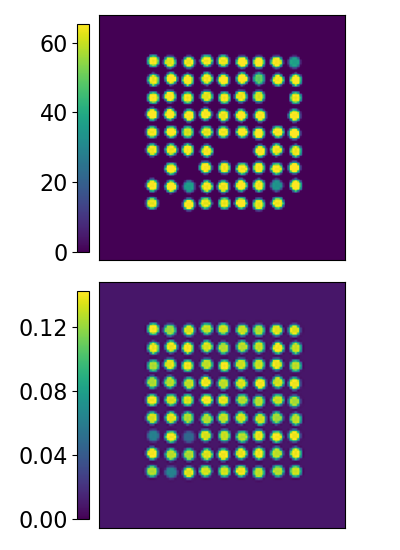}};
        \node[rotate=90] at ($(-1.8\imSz,-\imSz)$) {\small 'Hard' data};
        \node[above=-1mm of im, align=center, scale=.9] {Ground\\ truth};
        \foreach \iter [count=\x from 0] in {1,...,5}{
            \node[anchor=north] (out) at ($(\x*\imSkip, 0)$) {\includegraphics[width=\imSz]{images/fno/simulated/decay2_fno_model_s20_n_hard_10_samples-3__pass_-mode_iter\iter.png}};
            \node[above=-1mm of out, scale=.9] {Iter \iter};
        }
        \node[rotate=0] at ($(4.6*\imSkip,-0.5*\imSz)$) {$\lambda$};
        \node[rotate=0] at ($(4.6*\imSkip,-1.43*\imSz)$) {$\mu$};
    \end{scope}
    \begin{scope}[yshift=-2\imSz]
        \node[anchor=north] (im) at ($(-1.0\imSz,0)$) {\includegraphics[width=1.35\imSz]{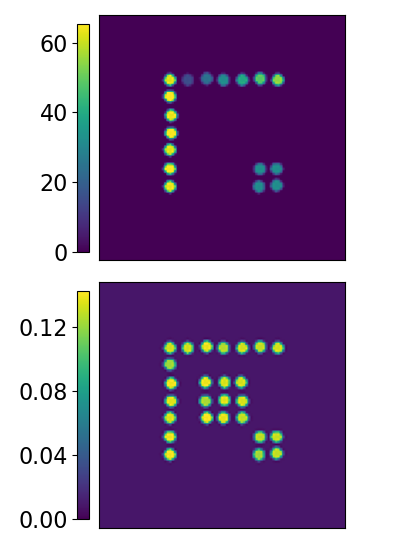}};
        \node[rotate=90] at ($(-1.8\imSz,-\imSz)$) {\small 'Extreme' data};
        \foreach \iter [count=\x from 0] in {1,...,5}{
            \node[anchor=north] (out) at ($(\x*\imSkip, 0)$) {\includegraphics[width=\imSz]{images/fno/simulated/decay2_fno_model_s20_n_crazy_data2-1__pass_-mode_iter\iter.png}};
        }
        \node[rotate=0] at ($(4.6*\imSkip,-0.5*\imSz)$) {$\lambda$};
        \node[rotate=0] at ($(4.6*\imSkip,-1.43*\imSz)$) {$\mu$};
    \end{scope}
    \end{tikzpicture}
    \caption{FNO iterates on one sample from 'Hard' and 'Extreme' test data sets. If the LM iterate is accepted in place of the model's proposition, the image is framed with red dashes.}
    \label{fig:fno_testdata}
\end{figure}

\begin{figure}
\setlength{\imSz}{0.16\columnwidth}
\setlength{\imSkip}{0.92\imSz}
    \centering
    \begin{tikzpicture}
    \begin{scope}
        \node[anchor=north] (im) at ($(-1.0\imSz,0)$) {\includegraphics[width=1.35\imSz]{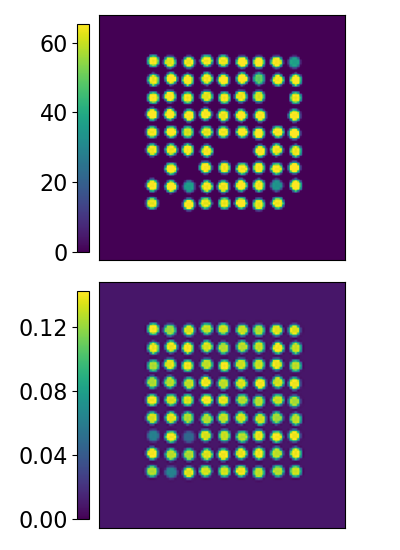}};
        \node[rotate=90] at ($(-1.8\imSz,-\imSz)$) {\small 'Hard' data};
        \node[above=-1mm of im, align=center, scale=.9] {Ground\\ truth};
        \foreach \iter [count=\x from 0] in {1,...,5}{
            \node[anchor=north] (out) at ($(\x*\imSkip, 0)$) {\includegraphics[width=\imSz]{images/wno/simulated/decay_wno_model_s12_n_hard_10_samples-3__pass_-mode_iter\iter.png}};
            \node[above=-1mm of out, scale=.9] {Iter \iter};
        }
        \node[rotate=0] at ($(4.6*\imSkip,-0.5*\imSz)$) {$\lambda$};
        \node[rotate=0] at ($(4.6*\imSkip,-1.43*\imSz)$) {$\mu$};
    \end{scope}
    \begin{scope}[yshift=-2\imSz]
        \node[anchor=north] (im) at ($(-1.0\imSz,0)$) {\includegraphics[width=1.35\imSz]{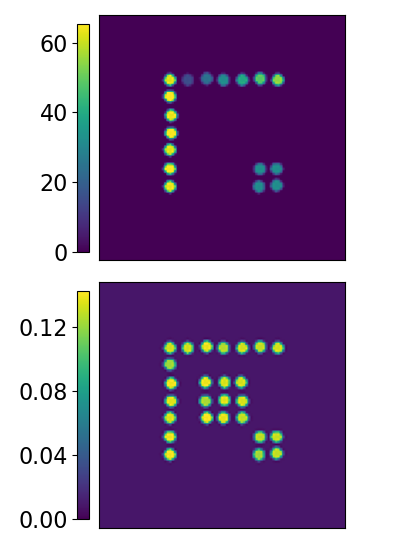}};
        \node[rotate=90] at ($(-1.8\imSz,-\imSz)$) {\small 'Extreme' data};
        \foreach \iter [count=\x from 0] in {1,...,5}{
            \node[anchor=north] (out) at ($(\x*\imSkip, 0)$) {\includegraphics[width=\imSz]{images/wno/simulated/decay_wno_model_s12_n_crazy_data2-1__pass_-mode_iter\iter.png}};
        }
        \node[rotate=0] at ($(4.6*\imSkip,-0.5*\imSz)$) {$\lambda$};
        \node[rotate=0] at ($(4.6*\imSkip,-1.43*\imSz)$) {$\mu$};
    \end{scope}
    \end{tikzpicture}
    \caption{WNO iterates on one sample from 'Hard' and 'Extreme' test data sets. If the LM iterate is accepted in place of the model's proposition, the image is framed with red dashes.}
    \label{fig:wno_testdata}
\end{figure}

All methods greatly accelerated the initial iterate of the 'hard' data which closely matches the training data. This first step already noticeably helps both the LM and the accelerated step that follow, which is evident looking at the plain LM iterates in \cref{fig:LMA_testdata}. Still, the attenuation images often correlate too strongly with the emission images and rods simulating fresh fuel (non-emitting but still attenuating) are often still missing, cf. the top- and bottom-right corner rods in \cref{fig:cnn_testdata,fig:fno_testdata,fig:wno_testdata}. 

\begin{figure}
\setlength{\imSz}{0.16\columnwidth}
\setlength{\imSkip}{0.92\imSz}
    \centering
    \begin{tikzpicture}
    \begin{scope}
        \foreach \iter [count=\x from 0] in {1,...,5,15}{
            \node[anchor=north] (out) at ($(\x*\imSkip, 0)$) {\includegraphics[width=\imSz]{images/LMA/LMA_n_hard_10_samples-3_iter\iter.png}};
            \node[above=-1mm of out, scale=.9] {Iter \iter};
        }
        \node[anchor=north] (cb) at ($(-0.6\imSz,0)$) {\includegraphics[width=0.33\imSz]{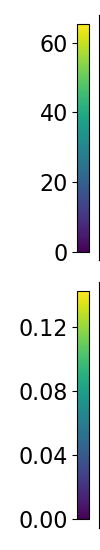}};
        \node[rotate=90] at ($(-.9\imSz,-\imSz)$) {\small 'Hard' data};
        \node[rotate=0] at ($(5.6*\imSkip,-0.5*\imSz)$) {$\lambda$};
        \node[rotate=0] at ($(5.6*\imSkip,-1.43*\imSz)$) {$\mu$};
    \end{scope}
    \begin{scope}[yshift=-2\imSz]
        \foreach \iter [count=\x from 0] in {1,...,5,15}{
            \node[anchor=north] (out) at ($(\x*\imSkip, 0)$) {\includegraphics[width=\imSz]{images/LMA/LMA_n_crazy_data2-1_iter\iter.png}};
        }
        \node[anchor=north] (cb) at ($(-0.6\imSz,0)$) {\includegraphics[width=0.33\imSz]{images/LMA/colorbar_simulated_data.png}};
        \node[rotate=90] at ($(-.9\imSz,-\imSz)$) {\small 'Extreme' data};
        \node[rotate=0] at ($(5.6*\imSkip,-0.5*\imSz)$) {$\lambda$};
        \node[rotate=0] at ($(5.6*\imSkip,-1.43*\imSz)$) {$\mu$};
    \end{scope}
    \end{tikzpicture}
    \caption{Traditional LM iterates on one sample from 'Hard' and 'Extreme' test data sets, to be used for comparison.}
    \label{fig:LMA_testdata}
\end{figure}

The handmade 'extreme' sample was more troublesome, the initial iterates were closer to those of the LM and especially the non-emitting rods in the middle were particularly difficult to reconstruct. Even though the relative $\ell^2$-errors are similar between emission and attenuation images, visually the attenuation images are of lower quality and some rods would likely be misclassified.

\subsection{Real data} \label{ssec:real data}

Thanks to STUK, we had access to an ample amount of real measurements from Finnish NPPs, in a variety of interesting configurations and assembly geometries. In \cref{fig:9x9_results,fig:atrium-results,fig:vver_results} the final (i.e. 5th) iterate obtained with the different methods is shown for two rectangular ($9\times 9$ and ATRIUM-10) and one hexagonal (VVER-440) fuel assemblies, respectively. For comparison, the LM-algorithm results after 15 and 5 iterations are also shown. Below the final emission and attenuation images are plots showing the average emission and attenuation values within a $3 \times 3$ pixel area at every possible rod location. These values and the true geometry are color coded based on the known rod types. All data are from the 650-700 keV range, which captures the $^{137}$Cs photopeak, has good photon statistics and low amount of noise and scattering \cite{virta2022improved}. In this energy range the noise profile is closer to Gaussian than Poisson.

\begin{figure}
\setlength{\imSz}{0.19\columnwidth}
\setlength{\imSkip}{0.92\imSz}
\centering
\begin{tikzpicture}
    \begin{scope}
        \node[anchor=north] (im) at ($(-2.0\imSz,0)$) {\includegraphics[width=\imSz]{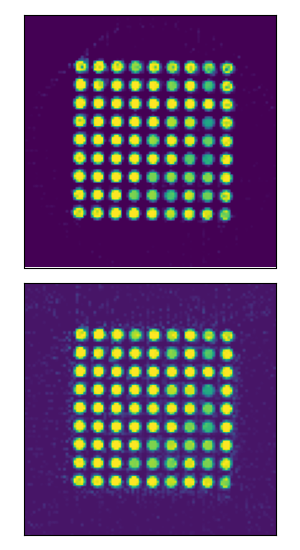}};
        \node[anchor=north] (im2) at ($(-1.1\imSz,0)$) {\includegraphics[width=\imSz]{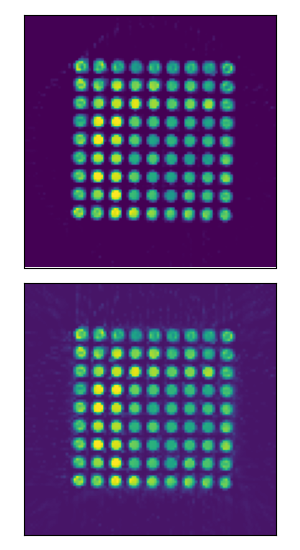}};
        \node[anchor=north] (cb) at ($(-2.58\imSz,-0.06\imSz)$) {\includegraphics[height=1.75\imSz]{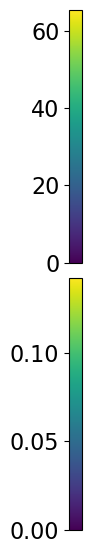}};
        \node[above=-1mm of im, align=center, scale=.9] {LM\\ (15~iters.)};
        \node[above=-1mm of im2, align=center, scale=.9] {LM\\ (5~iters.)};
        \foreach \fullmodel/\model [count=\x from 0] in {decay2_cnn_model_s6/CNN, decay2_fno_model_s20/FNO, decay_wno_model_s12/WNO}{
            \node[anchor=north] (out) at ($(\x*\imSkip, 0)$) {\includegraphics[width=\imSz]{images/KWU-9x9/\fullmodel_n_1018_OL21_KWU-9x9-5_E650_final.png}};
            
            \node[above=-1mm of out, scale=.9] {\model};
        }
        \node[rotate=0] at ($(2.6*\imSkip,-0.5*\imSz)$) {$\lambda$};
        \node[rotate=0] at ($(2.6*\imSkip,-1.43*\imSz)$) {$\mu$};
    \end{scope}
    \setlength{\imSz}{0.3\columnwidth}
    \setlength{\imSkip}{1.02\imSz}
    \begin{scope}[yshift=-1.5\imSz, xshift=-1.1\imSz]
        \node[anchor=north] (im3) at ($(2.1\imSkip,0.1)$) {\includegraphics[height=\imSz]{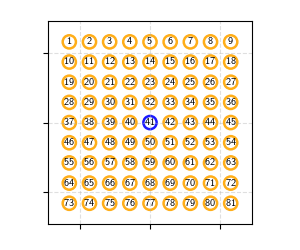}};
        \node[anchor=north] (im2) at ($(\imSkip,0)$) {\includegraphics[height=\imSz]{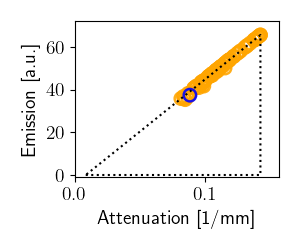}};
        \node[anchor=north] (im) at ($(0,0)$) {\includegraphics[height=\imSz]{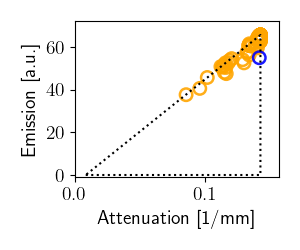}};
        \node[above=-3mm of im, align=center, scale=.9] {LM\\ (15~iters.)};
        \node[above=-3mm of im2, align=center, scale=.9] {LM\\ (5~iters.)};
        \node[above=-2mm of im3, align=center, scale=.9] {True geometry};
        \foreach \fullmodel/\model [count=\x from 0] in {decay_wno_model_s12/WNO, decay2_fno_model_s20/FNO, decay2_cnn_model_s6/CNN}{
            \node[anchor=north] (out) at ($(2*\imSkip - \x*\imSkip, -\imSkip)$) {\includegraphics[height=\imSz]{images/KWU-9x9/\fullmodel_n_1018_OL21_KWU-9x9-5_E650_classification.png}};
            
            \node[above=-3mm of out, scale=.9] {\model};
        }
        \node[anchor=north] (leg) at (\imSkip, -2\imSkip) {\includegraphics[width=2\imSz]{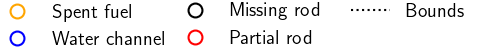}};
    \end{scope}
\end{tikzpicture}
\vspace*{-1em}
    \caption{Final results from the $9 \times 9$ data. Rows 1-2: final emission and attenuation images reconstructed with the traditional LM-algorithm and the different model-based accelerated methods. If the 5th LM iterate is accepted in place of the model's proposition, the image is framed with red dashes. Rows 3-4: average emission and attenuation value at each potential rod location and the true rod types.}
    \label{fig:9x9_results}
\end{figure}

The $9 \times 9$ data shown in \cref{fig:9x9_results} is very unassuming as no rods are missing or replaced. However, compared to the 15 iteration LM reconstructions, all of the accelerated results are closer to the real emission and attenuation levels, whereas after 5 iterations the results from LM are quite dim and the average values are quite spread out on the attenuation-emission plot. The average plots also show how all methods have problems reconstructing the central water channel properly. Even the LM-algorithm still misclassifies the central location after 15 iterations.

\begin{figure}
\setlength{\imSz}{0.19\columnwidth}
\setlength{\imSkip}{0.92\imSz}
\centering
\begin{tikzpicture}
    \begin{scope}
        \node[anchor=north] (im) at ($(-2.0\imSz,0)$) {\includegraphics[width=\imSz]{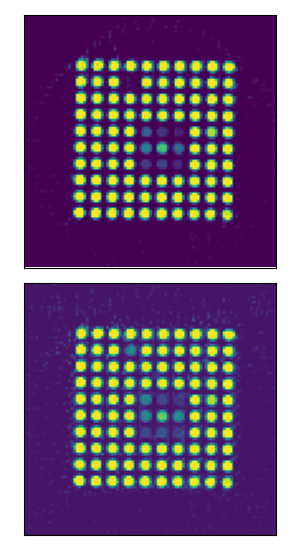}};
        \node[anchor=north] (im2) at ($(-1.1\imSz,0)$) {\includegraphics[width=\imSz]{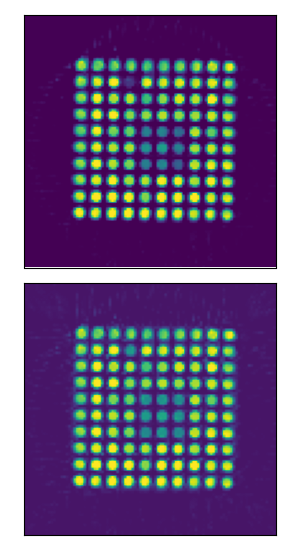}};
        \node[anchor=north] (cb) at ($(-2.58\imSz,-0.06\imSz)$) {\includegraphics[height=1.75\imSz]{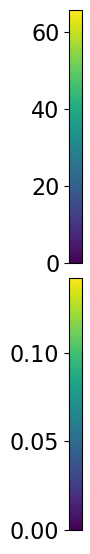}};
        \node[above=-1mm of im, align=center, scale=.9] {LM\\ (15~iters.)};
        \node[above=-1mm of im2, align=center, scale=.9] {LM\\ (5~iters.)};
        \foreach \fullmodel/\model [count=\x from 0] in {decay2_cnn_model_s6/CNN, decay2_fno_model_s20/FNO, decay_wno_model_s12/WNO}{
            \node[anchor=north] (out) at ($(\x*\imSkip, 0)$) {\includegraphics[width=\imSz]{images/ATRIUM-10/\fullmodel_n_2250_OL21_ATRIUM-10-2_E650_final.png}};

            \node[above=-1mm of out, scale=.9] {\model};
        }
        \node[rotate=0] at ($(2.6*\imSkip,-0.5*\imSz)$) {$\lambda$};
        \node[rotate=0] at ($(2.6*\imSkip,-1.43*\imSz)$) {$\mu$};
    \end{scope}
    \setlength{\imSz}{0.3\columnwidth}
    \setlength{\imSkip}{1.02\imSz}
    \begin{scope}[yshift=-1.5\imSz, xshift=-1.1\imSz]
        \node[anchor=north] (im3) at ($(2.1\imSkip,0.1)$) {\includegraphics[height=\imSz]{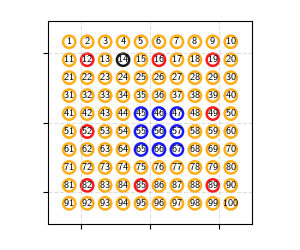}};
        \node[anchor=north] (im2) at ($(\imSkip,0)$) {\includegraphics[height=\imSz]{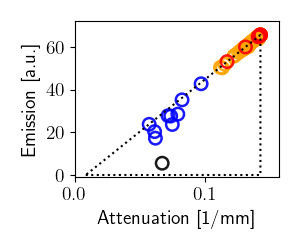}};
        \node[anchor=north] (im) at ($(0,0)$) {\includegraphics[height=\imSz]{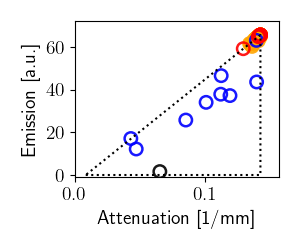}};
        \node[above=-3mm of im, align=center, scale=.9] {LM\\ (15~iters.)};
        \node[above=-3mm of im2, align=center, scale=.9] {LM\\ (5~iters.)};
        \node[above=-2mm of im3, align=center, scale=.9] {True geometry};
        \foreach \fullmodel/\model [count=\x from 0] in {decay_wno_model_s12/WNO, decay2_fno_model_s20/FNO, decay2_cnn_model_s6/CNN}{
            \node[anchor=north] (out) at ($(2*\imSkip - \x*\imSkip, -\imSkip)$) {\includegraphics[height=\imSz]{images/ATRIUM-10/\fullmodel_n_2250_OL21_ATRIUM-10-2_E650_classification.png}};
            
            \node[above=-3mm of out, scale=.9] {\model};
        }
        \node[anchor=north] (leg) at (\imSkip, -2\imSkip) {\includegraphics[width=2\imSz]{images/rod_legend.png}};
    \end{scope}
\end{tikzpicture}
\vspace*{-1em}
    \caption{Final results from the ATRIUM-10 data. Rows 1-2: final emission and attenuation images reconstructed with the traditional LM-algorithm and the different model-based accelerated methods. If the 5th LM iterate is accepted in place of the model's proposition, the image is framed with red dashes. Rows 3-4: average emission and attenuation value at each potential rod location and the true rod types.}
    \label{fig:atrium-results}
\end{figure}

The ATRIUM-10 data shown in \cref{fig:atrium-results} has a large $3\times3$ water channel, a missing rod on the second row and 8 partial-length rods. Although, all of the partial rods are likely in full view and indistinguishable from full rods at this measurement height. The CNN does a good job of separating missing and present rods and the quality is similar to the LM after 15 iterations. Interestingly both have trouble with the center of the water channel which is misclassified based on the averages. Visually both FNO and WNO have some trouble and artefacts can be seen near the water channel, although these do not seem to affect the FNO averages as the classification looks very good. All methods reconstruct the missing rod well, especially the lack of emissions.

\begin{figure}
\setlength{\imSz}{0.19\columnwidth}
\setlength{\imSkip}{0.92\imSz}
\centering
\begin{tikzpicture}
    \begin{scope}
        \node[anchor=north] (im) at ($(-2.0\imSz,0)$) {\includegraphics[width=\imSz]{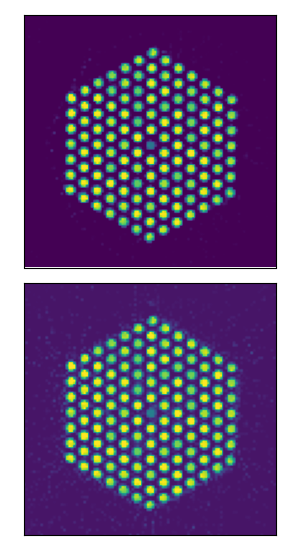}};
        \node[anchor=north] (im2) at ($(-1.1\imSz,0)$) {\includegraphics[width=\imSz]{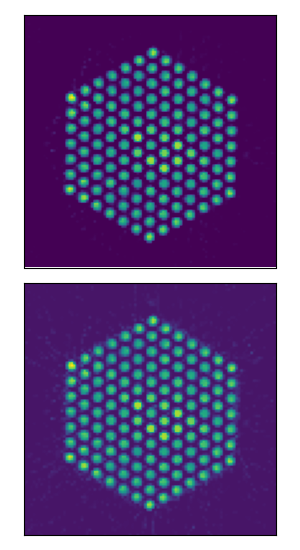}};
        \node[anchor=north] (cb) at ($(-2.58\imSz,-0.06\imSz)$) {\includegraphics[height=1.75\imSz]{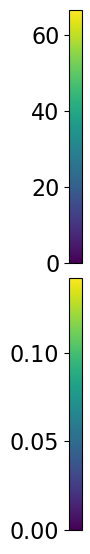}};
        \node[above=-1mm of im, align=center, scale=.9] {LM\\ (15~iters.)};
        \node[above=-1mm of im2, align=center, scale=.9] {LM\\ (5~iters.)};
        \foreach \fullmodel/\model [count=\x from 0] in {decay2_cnn_model_s6/CNN, decay2_fno_model_s20/FNO, decay_wno_model_s12/WNO}{
            \node[anchor=north] (out) at ($(\x*\imSkip, 0)$) {\includegraphics[width=\imSz]{images/VVER-440/\fullmodel_n_22432039_LO21_VVER-440_E650_final.png}};
            
            \node[above=-1mm of out, scale=.9] {\model};
        }
        \node[rotate=0] at ($(2.6*\imSkip,-0.5*\imSz)$) {$\lambda$};
        \node[rotate=0] at ($(2.6*\imSkip,-1.43*\imSz)$) {$\mu$};
    \end{scope}
    \setlength{\imSz}{0.3\columnwidth}
    \setlength{\imSkip}{1.02\imSz}
    \begin{scope}[yshift=-1.5\imSz, xshift=-1.1\imSz]
        \node[anchor=north] (im3) at ($(2.1\imSkip,0.1)$) {\includegraphics[height=\imSz]{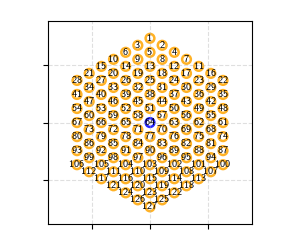}};
        \node[anchor=north] (im2) at ($(\imSkip,0)$) {\includegraphics[height=\imSz]{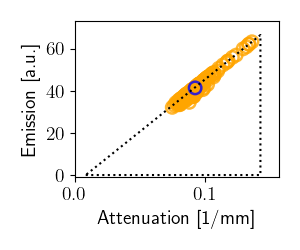}};
        \node[anchor=north] (im) at ($(0,0)$) {\includegraphics[height=\imSz]{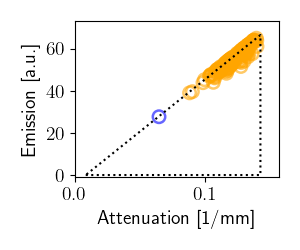}};
        \node[above=-3mm of im, align=center, scale=.9] {LM\\ (15~iters.)};
        \node[above=-3mm of im2, align=center, scale=.9] {LM\\ (5~iters.)};
        \node[above=-2mm of im3, align=center, scale=.9] {True geometry};
        \foreach \fullmodel/\model [count=\x from 0] in {decay_wno_model_s12/WNO, decay2_fno_model_s20/FNO, decay2_cnn_model_s6/CNN}{
            \node[anchor=north] (out) at ($(2*\imSkip - \x*\imSkip, -\imSkip)$) {\includegraphics[height=\imSz]{images/VVER-440/\fullmodel_n_22432039_LO21_VVER-440_E650_classification.png}};
            
            \node[above=-3mm of out, scale=.9] {\model};
        }
        \node[anchor=north] (leg) at (\imSkip, -2\imSkip) {\includegraphics[width=2\imSz]{images/rod_legend.png}};
    \end{scope}
\end{tikzpicture}
\vspace*{-1em}
    \caption{Final results from the VVER-440 data. Rows 1-2: final emission and attenuation images reconstructed with the traditional LM-algorithm and the different model-based accelerated methods. If the 5th LM iterate is accepted in place of the model's proposition, the image is framed with red dashes. Rows 3-4: average emission and attenuation value at each potential rod location and the true rod types.}
    \label{fig:vver_results}
\end{figure}

Finally, the hexagonal VVER-440 assembly shown in \cref{fig:vver_results} is reconstructed surprisingly well by all methods even though all of the training data matched only the rectangular $9\times 9$ assembly. The average emission and attenuation values seem to be slightly underestimated and none of the accelerated methods can clearly distinguish the central water channel and WNO in particular is showing high emission and attenuation values where none should be present.

\section{Conclusions} \label{sec:conclusions}

We proposed an iterative solver for Passive Gamma Emission Tomography that accelerates the Levenberg–Marquardt trust-region algorithm with a learned Deep Gauss-Newton update step. The scheme is hybrid by design: at each iteration the deterministic algorithm proposes an update step, a learned operator refines it, and a trust-region acceptance criterion decides which of the two is kept. Thanks to this acceptance check, we can guarantee that the iterates do not diverge and we can make informed choices about how and if to proceed with the iterations.

Within this framework we compared three architectures of similar parameter count: an encoder–decoder-style Convolutional Neural Network, Fourier Neural Operators, and Wavelet Neural Operators. We used only simulated data of the $9\times 9$ fuel assembly geometry for training, and with similar test data all three architectures improved the quality of the 15 iteration LM baseline within five accelerated iterations, the first one or two iterates in particular. On the harder handcrafted sample, the acceptance check failed more often and the tested architectures produced some visual artefacts, but still remained robust against out-of-distribution inputs or hallucinations.

The real data experiments revealed some architecture-specific trade-offs. The CNN and WNO generalized reasonably well even to unseen geometries, while the FNO was more sensitive to the real measurements and produced distinct smearing. However, this is unlikely to occur in practice since the general shape of the fuel assembly is visible from the outside and the regularization term of the LM-algorithm requires knowledge of the geometry anyway. None of the architectures dominated across the test cases, although the CNN architecture remains surprisingly effective and easy to tune in imaging applications. Both neural operator architectures are noticeably more abstract and implicit with many of the parameter choices such as rank, width and depth of the layers. Finally, the lack of real data or realistic noise and scattering in the training data likely impaired the performance of all three methods.

The seemingly good results with extreme simulated outliers did not directly transfer to robustness against the physical phenomena present in real measurements and all methods had problems with seemingly simple looking real data. However, seeing how the different fuel geometries (such as the hexagonal VVER-440) were all reconstructed to a similar level of quality, the method is robust in that regard.

Several development directions remain open. The convergence analysis could be extended to the fixed-point continuation scheme and changing regularization parameters, along the lines of \cite{fest2022fixed, Goldfarb2011, Hale2008}. The gains in computational efficiency of the proposed method could make it better suited for integration into statistical inference and uncertainty quantification frameworks, which are increasingly relevant as PGET moves toward routine use at the disposal facility. However, the deep Gauss-Newton scheme still requires computing traditional iterates which limits potential speedups. Finally, our current simulation pipeline has clear limitations and training on higher-fidelity Monte Carlo simulations or on curated real measurements should improve robustness and reconstruction quality.

The type and strength of prior information used in the PGET reconstruction process is also worth discussion. Here we use the relatively strong geometry-aware prior \cite{backholm2020simultaneous} to compute the iterates for the training data but once trained, the models seem to work well with different geometries and priors. In fact, we got better results when decreasing the strength of the regularization term in the traditional iterates. However, the original prior information and for example the particular choice of regularization parameters remains inherited in the models. In the future it would be better to make this more explicit to the user since relying on prior information is a sensitive topic in nuclear safeguards applications. An alternate approach would be to make the continuation method rigorous for LM as it would allow decreasing the effect of regularization over iterates.

\section*{Acknowledgments}
ToH, SH and TaH are supported by the Research Council of Finland (Flagship of Advanced Mathematics for Sensing, Imaging and Modelling, grant no. 359183).

SH is also supported by the Finnish Ministry of Education and Culture’s Pilot for Doctoral Programmes (Doctoral Education Pilot for Mathematics of Sensing, Imaging and Modelling (DREAM)).

The authors wish to thank P. Dendooven and H. Li for useful discussions and collaborative development of the PGET codebase, and both LUT Department of Computational Engineering and CSC – IT Center for Science, Finland, for computational assistance.

\bibliographystyle{ieeetran}
\bibliography{references.bib}

@inproceedings{white2019verification,
  title={Verification of spent nuclear fuel using passive gamma emission tomography ({PGET})},
  author={White, Timothy and Mayorov, Mikhail and Lebrun, Alain R. and Peura, Pauli},
  booktitle={IAEA Symposium on International Safeguards. Book of Abstracts},
  number={IAEA-CN--267},
  pages={198--198},
  year={2019}
}

@misc{posiva2024yjh,
  title={{YJH 2024 Olkiluodon ja Loviisan ydinlaitosten ydinjätehuollon ohjelma vuosille 2025-2027}},
  author={{Posiva~Oy}},
  pages={1-70},
  year={2024},
  howpublished={Posiva public reports and publications: \url{https://www.posiva.fi/material/sites/posivaraportit/20240210-1210-H4n8OESC1/yvbmsegfm/YJH-2024-ohjelma_web.pdf}},
  note={(In Finnish), referenced 5.9.2025.}
}

@inproceedings{white2018application,
  title={Application of passive gamma emission tomography ({PGET}) for the verification of spent nuclear fuel},
  author={White, Timothy and Mayorov, Mikhail and Lebrun, Alain and Peura, Pauli and Honkamaa, Tapani and Dahlberg, Joakim and Keubler, Jens and Ivanov, Victor and Turunen, Asko},
  booktitle={INMM 59th Annual Meeting, Baltimore, Maryland, USA},
  year={2018}
}

@article{belanger2018effect,
  title={Effect of gamma-ray energy on image quality in passive gamma emission tomography of spent nuclear fuel},
  author={B{\'e}langer-Champagne, Camille and Peura, Pauli and Eerola, Paula and Honkamaa, Tapani and White, Timothy and Mayorov, Mikhail and Dendooven, Peter},
  journal={IEEE transactions on nuclear science},
  volume={66},
  number={1},
  pages={487--496},
  year={2018},
  publisher={IEEE}
}

@article{backholm2020simultaneous,
  title={Simultaneous reconstruction of emission and attenuation in passive gamma emission tomography of spent nuclear fuel},
  author={Backholm, Rasmus and Bubba, Tatiana A. and B{\'e}langer-Champagne, Camille and Helin, Tapio and Dendooven, Peter and Siltanen, Samuli},
  journal={Inverse Problems \& Imaging},
  volume={14},
  number={2},
  year={2020}
}

@article{virta2020fuel,
  title={Fuel rod classification from passive gamma emission tomography ({PGET}) of spent nuclear fuel assemblies},
  author={Virta, Riina and Backholm, Rasmus and Bubba, Tatiana A. and Helin, Tapio and Moring, Mikael and Siltanen, Samuli and Dendooven, Peter and Honkamaa, Tapani},
  journal={ESARDA Bulletin},
  volume={2020},
  number={61},
  pages={10--21},
  year={2020},
  publisher={European Safeguards Research and Development Association}
}

@article{virta2022improved,
  title={Improved Passive Gamma Emission Tomography image quality in the central region of spent nuclear fuel},
  author={Virta, Riina and Bubba, Tatiana A. and Moring, Mikael and Siltanen, Samuli and Honkamaa, Tapani and Dendooven, Peter},
  journal={Scientific Reports},
  volume={12},
  number={1},
  pages={12473},
  year={2022},
  publisher={Nature Publishing Group UK London}
}

@phdthesis{virta2024gamma,
  title={Gamma tomography of spent nuclear fuel for geological repository safeguards},
  author={Virta, Riina},
  journal={HIP Internal Report},
  school={University of Helsinki},
  pages={60},
  year={2024},
  note={\url{http://hdl.handle.net/10138/575149}}
}

@article{herzberg2021graph,
  title={Graph convolutional networks for model-based learning in nonlinear inverse problems},
  author={Herzberg, William and Rowe, Daniel B. and Hauptmann, Andreas and Hamilton, Sarah J.},
  journal={IEEE transactions on computational imaging},
  volume={7},
  pages={1341--1353},
  year={2021},
  publisher={IEEE}
}

@article{mozumder2021model,
  title={A model-based iterative learning approach for diffuse optical tomography},
  author={Mozumder, Meghdoot and Hauptmann, Andreas and Nissil{\"a}, Ilkka and Arridge, Simon R. and Tarvainen, Tanja},
  journal={IEEE Transactions on Medical Imaging},
  volume={41},
  number={5},
  pages={1289--1299},
  year={2021},
  publisher={IEEE}
}

@article{hauptmann2018model,
  title={Model-based learning for accelerated, limited-view 3-D photoacoustic tomography},
  author={Hauptmann, Andreas and Lucka, Felix and Betcke, Marta and Huynh, Nam and Adler, Jonas and Cox, Ben and Beard, Paul and Ourselin, Sebastien and Arridge, Simon},
  journal={IEEE transactions on medical imaging},
  volume={37},
  number={6},
  pages={1382--1393},
  year={2018},
  publisher={IEEE}
}

@article{tripura2023wavelet,
  title={Wavelet neural operator for solving parametric partial differential equations in computational mechanics problems},
  author={Tripura, Tapas and Chakraborty, Souvik},
  journal={Computer Methods in Applied Mechanics and Engineering},
  volume={404},
  pages={115783},
  year={2023},
  publisher={Elsevier}
}

@inproceedings{honkamaa2014prototype,
  title={A prototype for passive gamma emission tomography},
  author={Honkamaa, Tapani and Levai, Ferenc and Turunen, Asko and Berndt, Reinhard and Vaccaro, Stefano and Schwalbach, Peter},
  booktitle={Proceedings of Symposium on International Safeguards},
  year={2014}
}

@book{kelley1999iterative,
  title={Iterative methods for optimization},
  author={Kelley, Carl T.},
  year={1999},
  publisher={SIAM},
  address={3600 Market Street, PA 19104, USA}
}

@book{engl1996regularization,
  title={Regularization of Inverse Problems},
  author={Engl, Heinz Werner and Hanke, Martin and Neubauer, A.},
  year={1996},
  isbn={978-0-7923-4157-4},
  publisher={Kluwer Academic Publisher},
  series={Mathematics and Its Applications},
  edition={1},
  address={3300 AA Dordrecht, The Netherlands}
}

@article{bonettini2008scaled,
  title={A scaled gradient projection method for constrained image deblurring},
  author={Bonettini, Silvia and Zanella, Riccardo and Zanni, Luca},
  journal={Inverse problems},
  volume={25},
  number={1},
  pages={015002},
  year={2008},
  publisher={IOP Publishing}
}

@mastersthesis{heikkinen2024model,
  title={Model-based imaging of spent nuclear fuel with passive gamma emission tomography},
  author={Heikkinen, Sara},
  year={2024},
  note = {\url{https://urn.fi/URN:NBN:fi-fe2024112797098}},
  school = {LUT University},
  type = {Master's thesis}
}

@article{li2020fourier,
  title={Fourier neural operator for parametric partial differential equations},
  author={Li, Zongyi and Kovachki, Nikola and Azizzadenesheli, Kamyar and Liu, Burigede and Bhattacharya, Kaushik and Stuart, Andrew and Anandkumar, Anima},
  journal={arXiv preprint arXiv:2010.08895},
  year={2020},
  note={Published as a conference paper at ICLR 2021}
}

@book{gonzalez2018digital,
  title={Digital image processing},
  subtitle={Fourth Edition},
  author={Gonzalez, Rafael C. and Woods, Richard E.},
  year={2018},
  publisher={Pearson Education},
  ISBN={978-0-13-335672-4},
  address={330 Hudson Street, NY 10013, USA}
}

@book{mallat1999wavelet,
  title={A wavelet tour of signal processing},
  author={Mallat, Stephane},
  year={1999},
  publisher={Academic Press},
  address={Burlington, MA 01803, USA}
}

@article{zhao2024review,
  title={A review of convolutional neural networks in computer vision},
  author={Zhao, Xia and Wang, Limin and Zhang, Yufei and Han, Xuming and Deveci, Muhammet and Parmar, Milan},
  journal={Artificial Intelligence Review},
  volume={57},
  number={4},
  pages={99},
  year={2024},
  publisher={Springer}
}

@misc{kossaifi2024neural,
   title={A Library for Learning Neural Operators},
   author={Kossaifi, Jean and Kovachki, Nikola and Li, 
   Zongyi and Pitt, David and
   Liu-Schiaffini, Miguel and Joseph George, Robert and
   Bonev, Boris and Azizzadenesheli, Kamyar and
   Berner, Julius and Anandkumar, Anima},
   year={2024},
   journal={arXiv preprint arXiv:2412.10354}
}

@article{kovachki2021neural,
    author = {Kovachki, Nikola and Li, Zongyi and Liu, Burigede and Azizzadenesheli, Kamyar and Bhattacharya, Kaushik and Stuart, Andrew and Anandkumar, Anima},
    title = {Neural operator: learning maps between function spaces with applications to {PDE}s},
    year = {2023},
    publisher = {JMLR.org},
    volume = {24},
    number = {1},
    issn = {1532-4435},
    journal = {J. Mach. Learn. Res.},
    month = jan
}

@article{hendrycks2016gaussian,
  title={Gaussian error linear units (gelus)},
  author={Hendrycks, Dan and Gimpel, Kevin},
  journal={arXiv preprint arXiv:1606.08415},
  year={2016}
}

@book{natterer2001mathematical,
  title={Mathematical methods in image reconstruction},
  author={Natterer, Frank and W{\"u}bbeling, Frank},
  year={2001},
  publisher={SIAM},
  address={3600 Market Street, PA 19104, USA}
}

@article{andrychowicz2016learning,
  title={Learning to learn by gradient descent by gradient descent},
  author={Andrychowicz, Marcin and Denil, Misha and Gomez, Sergio and Hoffman, Matthew W. and Pfau, David and Schaul, Tom and Shillingford, Brendan and De Freitas, Nando},
  journal={Advances in neural information processing systems},
  volume={29},
  year={2016}
}

@inproceedings{venkatakrishnan2013plug,
  title={Plug-and-play priors for model based reconstruction},
  author={Venkatakrishnan, Singanallur V. and Bouman, Charles A. and Wohlberg, Brendt},
  booktitle={2013 IEEE global conference on signal and information processing},
  pages={945--948},
  year={2013},
  organization={IEEE}
}

@article{stone1998convolution,
  title={Convolution theorems for linear transforms},
  author={Stone, Harold S.},
  journal={IEEE transactions on signal processing},
  volume={46},
  number={10},
  pages={2819--2821},
  year={1998},
  publisher={IEEE}
}

@article{lanthaler2025nonlocality,
  title={Nonlocality and nonlinearity implies universality in operator learning},
  author={Lanthaler, Samuel and Li, Zongyi and Stuart, Andrew M.},
  journal={Constructive Approximation},
  pages={1--43},
  year={2025},
  publisher={Springer}
}

@article{fahy2024greedy,
  title={Greedy Learning to Optimize with Convergence Guarantees},
  author={Fahy, Patrick and Golbabaee, Mohammad and Ehrhardt, Matthias J.},
  journal={arXiv preprint arXiv:2406.00260},
  year={2024}
}

@inproceedings{li2023learning,
  title={Learning preconditioners for conjugate gradient PDE solvers},
  author={Li, Yichen and Chen, Peter Yichen and Du, Tao and Matusik, Wojciech},
  booktitle={International Conference on Machine Learning},
  pages={19425--19439},
  year={2023},
  organization={PMLR}
}

@article{koolstra2022learning,
  title={Learning a preconditioner to accelerate compressed sensing reconstructions in {MRI}},
  author={Koolstra, Kirsten and Remis, Rob},
  journal={Magnetic Resonance in Medicine},
  volume={87},
  number={4},
  pages={2063--2073},
  year={2022},
  publisher={Wiley Online Library}
}

@article{liao2022learning,
  title={Learning to optimize quasi-Newton methods},
  author={Liao, Isaac and Dangovski, Rumen R. and Foerster, Jakob N. and Solja{\v{c}}i{\'c}, Marin},
  journal={arXiv preprint arXiv:2210.06171},
  year={2022}
}

@article{chen2022learning,
  title={Learning to optimize: A primer and a benchmark},
  author={Chen, Tianlong and Chen, Xiaohan and Chen, Wuyang and Heaton, Howard and Liu, Jialin and Wang, Zhangyang and Yin, Wotao},
  journal={Journal of Machine Learning Research},
  volume={23},
  number={189},
  pages={1--59},
  year={2022}
}

@article{monga2019algorithm,
  title={Algorithm unrolling: interpretable, efficient deep learning for signal and image processing},
  author={Monga, Vishal and Li, Yuelong and Eldar, Yonina C.},
  journal={arXiv preprint arXiv:1912.10557},
  year={2019}
}

@book{Dive_into_DL,
    title={Dive into Deep Learning},
    author={Zhang, Aston and Lipton, Zachary C. and Li, Mu and Smola, Alexander J.},
    publisher={Cambridge University Press},
    note={\url{https://D2L.ai}},
    year={2023},
    address={Shaftesbury Road, CB2 8EA, England}
}

@article{leCunConv,
  title={Backpropagation Applied to Handwritten Zip Code Recognition},
  author={LeCun, Y. and Boser, B. and Denker, J. S. and Henderson, D. and Howard, R.E. and Hubbard, W. and Jackel, L.D.},
  journal={Neural Computation},
  volume={1},
  issue={4},
  pages={541--551},
  year={1989},
  issn={0899-7667}
}

@article{encoder_decoder_architechture,
  author={Badrinarayanan, Vijay and Kendall, Alex and Cipolla, Roberto},
  journal={IEEE Transactions on Pattern Analysis and Machine Intelligence}, 
  title={SegNet: A Deep Convolutional Encoder-Decoder Architecture for Image Segmentation}, 
  year={2017},
  volume={39},
  pages={2481-2495},
  number={12},
  doi={10.1109/TPAMI.2016.2644615},
}

@book{Deep_learning_Python_book,
  title={Deep Learning with Python},
  subtitle={Learn Best Practices of Deep Learning Models with PyTorch},
  author={Ketkar, Nikhil and Moolayil, Jojo},
  year={2021},
  isbn={978-1-4842-5364-9},
  publisher={Apress Media},
  edition={2},
  address={One New York Plaza, NY 10004-1562, USA}
}

@inproceedings{PyTorch,
    author = {Ansel, Jason and Yang, Edward and He, Horace and Gimelshein, Natalia and Jain, Animesh and Voznesensky, Michael and Bao, Bin and Bell, Peter and Berard, David and Burovski, Evgeni and Chauhan, Geeta and Chourdia, Anjali and Constable, Will and Desmaison, Alban and DeVito, Zachary and Ellison, Elias and Feng, Will and Gong, Jiong and Gschwind, Michael and Hirsh, Brian and Huang, Sherlock and Kalambarkar, Kshiteej and Kirsch, Laurent and Lazos, Michael and Lezcano, Mario and Liang, Yanbo and Liang, Jason and Lu, Yinghai and Luk, C. K. and Maher, Bert and Pan, Yunjie and Puhrsch, Christian and Reso, Matthias and Saroufim, Mark and Siraichi, Marcos Yukio and Suk, Helen and Zhang, Shunting and Suo, Michael and Tillet, Phil and Zhao, Xu and Wang, Eikan and Zhou, Keren and Zou, Richard and Wang, Xiaodong and Mathews, Ajit and Wen, William and Chanan, Gregory and Wu, Peng and Chintala, Soumith},
    title = {PyTorch 2: Faster Machine Learning Through Dynamic Python Bytecode Transformation and Graph Compilation},
    year = {2024},
    isbn = {9798400703850},
    publisher = {Association for Computing Machinery},
    address = {NY, USA},
    doi = {10.1145/3620665.3640366},
    booktitle = {Proceedings of the 29th ACM International Conference on Architectural Support for Programming Languages and Operating Systems, Volume 2},
    pages = {929–947},
    numpages = {19},
    location = {La Jolla, CA, USA},
    series = {ASPLOS '24}
}

@misc{WNO_code,
    title = {{Wavelet-Neural-Operator (WNO)}},
    author = {Tripura, Tapas and Chakraborty, Souvik},
    howpublished = {Github repository: \url{https://github.com/TapasTripura/WNO}},
    note = {(v2.0.0) Referenced in 22.10.2025}
}

@article{takamoto2022pdebench,
  title={PDEBench: An extensive benchmark for scientific machine learning},
  author={Takamoto, Makoto and Praditia, Timothy and Leiteritz, Raphael and MacKinlay, Daniel and Alesiani, Francesco and Pfl{\"u}ger, Dirk and Niepert, Mathias},
  journal={Advances in Neural Information Processing Systems},
  volume={35},
  pages={1596--1611},
  year={2022}
}

@article{liu2025enhancing,
  title={Enhancing fourier neural operators with local spatial features},
  author={Liu, Chaoyu and Murari, Davide and Liu, Lihao and Li, Yangming and Budd, Chris and Sch{\"o}nlieb, Carola-Bibiane},
  journal={arXiv preprint arXiv:2503.17797},
  year={2025}
}

@article{raonic2023convolutional,
  title={Convolutional neural operators for robust and accurate learning of {PDE}s},
  author={Raonic, Bogdan and Molinaro, Roberto and De Ryck, Tim and Rohner, Tobias and Bartolucci, Francesca and Alaifari, Rima and Mishra, Siddhartha and de B{\'e}zenac, Emmanuel},
  journal={Advances in Neural Information Processing Systems},
  volume={36},
  pages={77187--77200},
  year={2023}
}

@article{dai2023neural,
  title={Neural operator learning for ultrasound tomography inversion},
  author={Dai, Haocheng and Penwarden, Michael and Kirby, Robert M. and Joshi, Sarang},
  journal={arXiv preprint arXiv:2304.03297},
  year={2023}
}

@inproceedings{kabri2023resolution,
  title={Resolution-invariant image classification based on Fourier neural operators},
  author={Kabri, Samira and Roith, Tim and Tenbrinck, Daniel and Burger, Martin},
  booktitle={International Conference on Scale Space and Variational Methods in Computer Vision},
  pages={236--249},
  year={2023},
  organization={Springer}
}

@article{cavallini2023vanquishing,
  title={Vanquishing the computational cost of passive gamma emission tomography simulations leveraging physics-aware reduced order modeling},
  author={Cavallini, Nicola and Ferretti, Riccardo and Bostrom, Gunnar and Croft, Stephen and Fassi, Aurora and Mercurio, Giovanni and Nonneman, Stefan and Favalli, Andrea},
  journal={Scientific Reports},
  volume={13},
  number={1},
  pages={15034},
  year={2023},
  publisher={Nature Publishing Group UK London}
}

@inproceedings{miller2018assessing,
  title={Assessing instrument performance for passive gamma emission tomography of spent fuel},
  author={Erin, Miller and Mozin, Vladimir and White, Timothy and Campbell, Luke W. and Wittman, Richard S. and Peura, Pauli},
  booktitle={INMM 59th Annual Meeting Paper Advanced Nondestructive Assay Techniques for Fuel Assemblies},
  year={2018}
}

@article{stefanov2014identification,
  title={The identification problem for the attenuated X-ray transform},
  author={Stefanov, Plamen},
  journal={American Journal of Mathematics},
  volume={136},
  number={5},
  pages={1215--1247},
  year={2014},
  publisher={Johns Hopkins University Press}
}

@article{dicken1999new,
  title={A new approach towards simultaneous activity and attenuation reconstruction in emission tomography},
  author={Dicken, Volker},
  journal={Inverse Problems},
  volume={15},
  number={4},
  pages={931--960},
  year={1999}
}

@article{hanke2010regularizing,
  title={The regularizing Levenberg-Marquardt scheme is of optimal order},
  author={Hanke, Martin},
  journal={The journal of integral equations and applications},
  pages={259--283},
  year={2010},
  publisher={JSTOR}
}

@book{nocedal2006numerical,
  title={Numerical optimization},
  author={Nocedal, Jorge and Wright, Stephen J.},
  year={2006},
  publisher={Springer},
  edition={2},
  isbn={978-0387-30303-1},
  address={233 Spring Street, NY 10013, USA}
}

@inproceedings{fest2022fixed,
  title={On a fixed-point continuation method for a convex optimization problem},
  author={Fest, Jean-Baptiste and Heikkil{\"a}, Tommi and Loris, Ignace and Martin, S{\'e}gol{\`e}ne and Ratti, Luca and Rebegoldi, Simone and Sarnighausen, Gesa},
  booktitle={INdAM Workshop: Advanced Techniques in Optimization for Machine learning and Imaging},
  pages={15--30},
  year={2022},
  organization={Springer}
}

@Article{Hale2008,
  author    = {Hale, Elaine and Yin, Wotao and Zhang, Yin},
  journal   = {{SIAM} Journal on Optimization},
  title     = {Fixed-Point Continuation for $\ell_1$-Minimization: Methodology and Convergence},
  year      = {2008},
  month     = {jan},
  number    = {3},
  pages     = {1107--1130},
  volume    = {19},
  doi       = {10.1137/070698920},
  publisher = {Society for Industrial {\&} Applied Mathematics ({SIAM})},
  school    = {Rice University}
}

@Article{Goldfarb2011,
  author    = {Goldfarb, Donald and Ma, Shiqian},
  journal   = {Foundations of Computational Mathematics},
  title     = {Convergence of Fixed-Point Continuation Algorithms for Matrix Rank Minimization},
  year      = {2011},
  month     = {feb},
  number    = {2},
  pages     = {183--210},
  volume    = {11},
  doi       = {10.1007/s10208-011-9084-6},
  publisher = {Springer Science and Business Media {LLC}},
}

@article{liu2024neural,
  title={Neural operators with localized integral and differential kernels},
  author={Liu-Schiaffini, Miguel and Berner, Julius and Bonev, Boris and Kurth, Thorsten and Azizzadenesheli, Kamyar and Anandkumar, Anima},
  journal={arXiv preprint arXiv:2402.16845},
  year={2024}
}

@inproceedings{heaton2023safeguarded,
  title={Safeguarded learned convex optimization},
  author={Heaton, Howard and Chen, Xiaohan and Wang, Zhangyang and Yin, Wotao},
  booktitle={Proceedings of the AAAI Conference on Artificial Intelligence},
  volume={37},
  number={6},
  pages={7848--7855},
  year={2023}
}

@article{kamilov2023plug,
  title={Plug-and-play methods for integrating physical and learned models in computational imaging: Theory, algorithms, and applications},
  author={Kamilov, Ulugbek S. and Bouman, Charles A. and Buzzard, Gregery T. and Wohlberg, Brendt},
  journal={IEEE Signal Processing Magazine},
  volume={40},
  number={1},
  pages={85--97},
  year={2023},
  publisher={IEEE}
}

@inproceedings{ryu2019plug,
  title={Plug-and-play methods provably converge with properly trained denoisers},
  author={Ryu, Ernest and Liu, Jialin and Wang, Sicheng and Chen, Xiaohan and Wang, Zhangyang and Yin, Wotao},
  booktitle={International Conference on Machine Learning},
  pages={5546--5557},
  year={2019},
  organization={PMLR}
}

@article{pesquet2021learning,
  title={Learning maximally monotone operators for image recovery},
  author={Pesquet, Jean-Christophe and Repetti, Audrey and Terris, Matthieu and Wiaux, Yves},
  journal={SIAM Journal on Imaging Sciences},
  volume={14},
  number={3},
  pages={1206--1237},
  year={2021},
  publisher={SIAM}
}

@article{bredies2024learning,
  title={Learning firmly nonexpansive operators},
  author={Bredies, Kristian and Chirinos-Rodriguez, Jonathan and Naldi, Emanuele},
  journal={arXiv preprint arXiv:2407.14156},
  year={2024}
}

@article{premont2022simple,
  title={A simple guard for learned optimizers},
  author={Pr{\'e}mont-Schwarz, Isabeau and V{\'i}tk{\r u}, Jaroslav and Feyereisl, Jan},
  journal={arXiv preprint arXiv:2201.12426},
  year={2022}
}

@inproceedings{
loshchilov2018decoupled,
title={Decoupled Weight Decay Regularization},
author={Ilya Loshchilov and Frank Hutter},
booktitle={International Conference on Learning Representations},
year={2019},
}

\end{document}